\documentclass[12pt]{QT}

\usepackage{graphicx}
\usepackage{amsmath}
\usepackage{amscd}
\usepackage{amssymb}
\usepackage{verbatim}
\usepackage[mathscr]{eucal}
\usepackage[all]{xy}
\usepackage{setspace}
\usepackage{psfrag}
\usepackage{url}
\usepackage{pinlabel}
\usepackage{float}

\usepackage[colorlinks=true,linkcolor=blue,citecolor=blue,urlcolor=blue]{hyperref}

\address[ina@rice.edu]{Ina Petkova, Department of Mathematics, Rice University, MS-136, Houston, TX 77005}

\title[Cables of thin knots]{Cables of Thin Knots and Bordered Heegaard Floer Homology}
\author{Ina Petkova}

\theoremstyle{plain}
\newtheorem{theorem}{Theorem}

\newtheorem{lemma}[theorem]{Lemma}
\newtheorem{corollary}[theorem]{Corollary}

\newtheorem{question}[theorem]{Question}

\newtheorem{defn}[theorem]{Definition}

\def\note{{\bf {\bigskip}{\noindent}Note: }}

\def\ackn{{\bf {\bigskip}{\noindent}Acknowledgments. }}

\def\bar{\overline}

\newcommand{\Z}{\mathbb Z}
\newcommand{\R}{\mathbb R}

\newcommand{\F}{\mathbb F}

\newcommand{\cf}{\mathit{CF}}
\newcommand{\hf}{\mathit{HF}}
\newcommand{\cfd}{\mathit{CFD}}
\newcommand{\cfa}{\mathit{CFA}}
\newcommand{\cfk}{\mathit{CFK}}
\newcommand{\hfk}{\mathit{HFK}}
\newcommand{\cfhat}{\widehat{\cf}}
\newcommand{\hfhat}{\widehat{\hf}}
\newcommand{\cfdhat}{\widehat{\cfd}}
\newcommand{\cfahat}{\widehat{\cfa}}
\newcommand{\cfkhat}{\widehat{\cfk}}
\newcommand{\hfkhat}{\widehat{\hfk}}
\newcommand{\cfkm}{\cfk^-}
\newcommand{\hfkm}{\hfk^-}
\newcommand{\cfam}{\cfa^-}

\newcommand{\h}{\langle h \rangle }
\newcommand{\g}{\langle g \rangle }

\newcommand{\HH}{\mathcal{H}}

\newcommand{\zz}{\mathcal Z}

\newcommand{\bdy}{\partial}

\begin{document}

\maketitle

\begin{abstract}
We use bordered Floer homology to give a formula for $\hfkhat(K_{p, pn+1})$ of any $(p, pn+1)$-cable of a thin knot $K$ in terms of $\Delta_K(t)$, $\tau(K)$, $p$, and $n$. We also give a formula for the Ozsv\'ath-Szab\'o concordance invariant $\tau(K_{p, q})$ in terms of $\tau(K)$, $p$, and $q$, for all relatively prime $p$ and $q$.
\end{abstract}




\section{Introduction}\label{intro}

In \cite{hfk}, Ozsv\'ath and Szab\'o introduce a powerful knot invariant using Heegaard diagrams. In this paper, we study its simplest version, the knot Floer homology $\hfkhat(K)$, which has the structure of a bigraded vector space over $\F_2$, the field with two elements. Its Euler characteristic is the symmetrized Alexander polynomial $\Delta_K(T)$, in the sense that
$$\sum_{i,j}  (-1)^i T^j  \textrm{ rank } \hfkhat_i(K, j)= \Delta_K(T).$$
The indices $i$  and $j$ in the summation stand for the \emph{Maslov grading} $M$ and the \emph{Alexander grading} $A$, respectively. It is sometimes convenient to make use of a third grading, $\delta = A-M$.

Originally, knot Floer homology was defined by counting pseudo-holomorphic curves in the $g$-fold symmetric product of a genus $g$ Heegaard surface. Later, combinatorial versions appeared, including a method using grid diagrams \cite{mos}. The complex coming from a grid diagram has $n!$ generators, where $n$ is the arc index of the knot, so this method only works well in practice for knots with few crossings or for special families of knots. In this paper, we instead use bordered Floer homology, which generalizes Heegaard Floer homology to $3$-manifolds with boundary, and to knots in $3$-manifolds with boundary \cite{bfh2}. The beauty of this theory is that it allows us to compute invariants for a space by cutting the space into simpler pieces, and studying the pieces and their gluing instead. This approach is particularly well suited for studying knot satellites. It was used by Levine to study generalizations of  Bing and Whitehead doubles \cite{lev9, lev10}. Here, we apply the bordered method to cables of thin knots. Our Corollary \ref{pq} has  since been generalized by Hom to cables of all knots \cite{homthesis}.

Let $K$ be a knot in $S^3$. Recall that the $(p,q)$-cable of $K$, denoted $K_{p,q}$, is the satellite knot with pattern the torus knot $T_{p,q}$ and companion $K$. In other words, if $T_{p,q}$ is drawn on the surface of an unknotted solid torus, then we obtain  $K_{p,q}$ by gluing the solid torus to the complement of $K$, identifying its meridian and preferred longitude with the meridian and preferred longitude of $K$.  Thus, $p$ and $q$ refer to the winding of the cable in the longitudinal and meridional directions of $K$, respectively. 

A knot $K$ is called \emph{Floer homologically thin} \cite{qalinks} if its knot Floer homology is supported in a single $\delta$-grading. Throughout this paper we will say \emph{thin} to mean Floer homologically thin. If the homology is supported on the diagonal $\delta = -\sigma/2$, where $\sigma$ denotes the knot signature, then we say the knot is \emph{$\sigma$-thin}, or  \emph{perfect} \cite{jr2br}. The class of $\sigma$-thin knots contains as a proper subset all
 quasi-alternating knots \cite{qalinks}, and in particular all alternating knots \cite{alt}.
 
Using the knot filtration on $\cfhat(S^3)$, Ozsv\'ath and Szab\'o define an integer knot invariant $\tau$ \cite{tau}, independently discovered by Rasmussen \cite{jrth}, whose absolute value is a lower bound on the four-ball genus. The behavior of $\tau$ under various satellite operations, such as cabling, Bing, and Whitehead doubling, has been studied extensively in recent years \cite{mc2, ln, cvc, lr9, lr9e, lev9, lev10}. In \cite{mc2}, Hedden gives upper and lower bounds for $\tau(K_{p, pn+1})$ in terms of $\tau(K)$, $p$, and $n$, and, for sufficiently large $|n|$, describes the knot Floer homology of the cable in the topmost Alexander gradings. In the case where $K$ is thin, we extend Hedden's results to a complete description of the knot Floer homology of the cable.
In particular, we derive a formula for $\hfkhat(K_{p, pn+1})$ and for $\tau(K_{p, pn+1})$ in terms of $\tau(K)$, $\Delta_K(t)$, $p$, and $n$. Note that for a $\sigma$-thin knot, $\tau(K) = -\sigma(K)/2$.

\begin{theorem}\label{main}
Suppose $K$ is a thin knot, and fix integers $p>1$ and $n$. Then $\tau(K)$, $\Delta_K(t)$, $p$, and $n$  determine the associated graded complex to $\cfkm(K_{p, pn+1})$, and in particular $\hfkhat(K_{p, pn+1})$. The complete description is given in Section \ref{tensor}. Further,
\begin{displaymath}
\tau(K_{p, pn+1}) = \left\{ \begin{array}{ll}
p\tau(K)+\frac{np(p-1)}{2} \quad & \textrm{if  $\tau(K)=0$ and $n\geq 0$, or if $\tau(K)>0$}\\
 p\tau(K)+\frac{np(p-1)}{2}+p-1 & \textrm{otherwise.}
\end{array} \right.
\end{displaymath}
\end{theorem}

We prove Theorem \ref{main} using the Pairing Theorem $11.21$ of \cite{bfh2}. 
Our method can easily be adapted to compute $\hfkhat(K_{p,q})$ for any relatively prime $p$ and $q$, as we explain at the end of Section \ref{cfa}.

We are grateful to Cornelia Van Cott for pointing out the following:

\begin{corollary}[Van Cott]\label{pq} 
Suppose $K$ is a thin knot, and $p$ and $q$ are relatively prime integers, with $p>0$. Then
\begin{displaymath}
\tau(K_{p,q}) = \left\{ \begin{array}{ll}
p\tau(K)+\frac{(p-1)(q-1)}{2} \quad & \textrm{if $\tau(K)=0$ and $q > 0$, or if $\tau(K)>0$}\\
p\tau(K)+\frac{(p-1)(q+1)}{2}& \textrm{otherwise.}
\end{array} \right.
\end{displaymath}
\end{corollary}

Note that since $K_{-p,-q} = -K_{p,q}$, where $-K_{p,q}$ is $K_{p,q}$ with reversed orientation, and since $\tau$ does not distinguish orientation, the result of Corollary \ref{pq} extends to all one-component cables of thin knots. 

Theorem \ref{main} and Corollary \ref{pq}, combined with the inequality $|\tau| \leq g_4$, provide  information about the four-ball genus of cables. For example

\begin{corollary}\label{genus} 
Suppose $K$ is a thin knot with $ g_4(K) = \tau(K) $,  $p>1$ is an integer, and $q>0$ is an integer relatively prime to $p$. Then $g_4(K_{p,q}) = \tau(K_{p,q})$. 
\end{corollary}

\note  A \emph{Mathematica} \cite{m} program implementing this method is available online   \cite{prog}. The program takes $\Delta_K(t)$, $\tau(K)$, $p$, and $n$ as input, and outputs the generators of $\hfkhat(K_{p, pn+1})$ as a list of ordered pairs of  Alexander and Maslov gradings. It then plots the result on  the $(A, M)$-lattice. The program computes $\hfkhat$ for cables with thousands of crossings in a matter of seconds, whereas the grid method would take billions of years. In Section \ref{ex} we give the result for the $(5, 16)$-cable of the knot $11n50$. This knot is interesting as it is the first known example of a homologically thin (with respect to $\hfkhat$, $\bar{\mathit{Kh}}$ and $\bar{\mathit{Kh}}'$), non-quasi-alternating knot \cite{nonqa}. 

 \ackn I wish to thank my advisor Peter Ozsv\'ath for suggesting this problem, and for his guidance. I am also thankful to Jon Bloom and Robert Lipshitz for many helpful conversations, and to Cornelia Van Cott for pointing out Corollary \ref{pq}. Last,  I thank the referees for their corrections and suggestions.


\section{Preliminaries on bordered Floer homology}\label{prelim}

We review the theory of bordered Floer homology, focusing on the special case of torus boundary, in  particular gluing a knot in the solid torus to a knot complement. For details and the more general theory, we refer the reader to \cite{bfh2}.

Let $Y$ be a closed oriented $3$-manifold, and let $F$ be a connected, oriented surface that separates $Y$ into two manifolds with boundary, $(Y_1, F)$ and $(Y_2, -F)$. Fix a Heegaard diagram $\HH = (\Sigma_g, \{\alpha_1, \ldots, \alpha_g\},  \{\beta_1, \ldots, \beta_g\}, z)$ for $Y$, and assume that $F$ intersects $\HH$ in a circle $\zz$ that contains the basepoint $z$ and crosses $2k$ $\alpha$-curves, twice each, and no $\beta$-curves. Assume also that no isotopies can be made to decrease the number of intersections with the $\alpha$-curves, and that $g(F) = k$. We call such $\zz$ a \emph{pointed matched circle}. Two points are matched if they belong to the same $\alpha$-arc. The circle $\zz$ separates $\HH$ into the two \emph{bordered Heegaard diagrams} $\mathcal H_1$ and $\mathcal H_2$, representing the two manifolds with boundary. The parametrization of the boundary is specified by $\bdy \mathcal H_1 = \zz$ or $\bdy \mathcal H_2 = -\zz$ respectively.

To a pointed matched circle $\zz$ we associate an $\mathcal A_{\infty}$-algebra $\mathcal A(\zz)$ over $\F_2$. To a bordered Heegaard diagram $(\mathcal H, z)=(\bar \Sigma, \boldsymbol{\bar\alpha}, \boldsymbol{\beta}, z)$, we associate either a left type $D$ structure $\cfdhat (\mathcal H, z)$ over $\mathcal A(-\bdy\mathcal H)$, or a right $\mathcal A_{\infty}$-module $\cfahat (\mathcal H, z)$ over $\mathcal A(\bdy\mathcal H)$. Similarly, we can represent a knot in a bordered $3$-manifold by a doubly-pointed bordered Heegaard diagram $(\mathcal H, z, w) = (\bar \Sigma, \boldsymbol{\bar\alpha}, \boldsymbol{\beta}, z, w)$, where $z$ and $w$ are in $\bar\Sigma\setminus(\boldsymbol{\bar\alpha}\cup \boldsymbol{\beta})$, and  $z\in \bdy \mathcal H$. To this diagram we can associate a right $\mathcal A_{\infty}$-module $\cfam(\mathcal H, z, w)$, this time over $\F_2[U]$, where a holomorphic curve passing through $w$ with multiplicity $n$ contributes $U^n$ to the multiplication. Setting $U=0$ gives  $\cfahat(\mathcal H, z, w)$, where we count only holomorphic curves that do not cross $w$.

From here on, let $Y_{p,1}$ stand for the $(p,1)$-cable in the $0$-framed solid torus, and let $Y_{K, n}$ be the $n$-framed knot complement $S^3\setminus K$, so that $Y_{p,1}\cup_{\bdy}Y_{K, n}$ is the pair $(S^3, K_{p, pn+1})$.  The separating surface $F = \bdy Y_{p,1}  = -\bdy Y_{K,n}$ is a torus, parametrized by the circle $\zz$ in Figure \ref{fig:zc}. The four $\alpha$-points divide the circle into the four upward-oriented arcs $\rho_0$, $\rho_1$, $\rho_2$, and $\rho_3$, where $\rho_0$ contains the basepoint $z$. The $\mathcal A_{\infty}$-algebra $\mathcal A$  is just a graded algebra; it has two idempotents, one for each $\alpha$-arc, and $6$ Reeb elements, coming from the Reeb chords $\rho_1, \rho_2$ and $\rho_3$ (see  \cite[Section 11.1]{bfh2}).  In  this case, $\cfdhat(Y_{K,n})$ can be derived explicitly from $\cfkm(K)$ (see Section \ref{dfrom-}), and is represented best using the coefficient maps $D_1, D_2, D_3, D_{12}, D_{23}$, or  $D_{123}$, which describe the sequence of Reeb chords that a holomorphic curve passes through. 

When at least one of $\cfa$ or $\cfd$ is bounded \cite[Definititions 2.4 and 2.22]{bfh2}, there is a particularly simple description for their $\mathcal A_{\infty}$ tensor product and the tensor differential. A product $a\boxtimes d$ is nonzero in $\cfa\boxtimes \cfd$ whenever $a$ and $d$ occupy complementary sets of $\alpha$-arcs.  The differential  $\bdy^{\boxtimes}(a_1\boxtimes d_1)$ has $a_2\boxtimes d_2$ in the image  whenever there is a sequence of coefficient maps $D_{I_1}, \ldots, D_{I_n}$ from $d_1$ to $d_2$ and a multiplication map $m_{n+1}(a_1, \rho_{I_1}, \ldots, \rho_{I_n})$ with $a_2$ in the image, both indexed the same way. See  \cite[Definition 2.26 and Equation (2.29)]{bfh2}.

\psfrag{a1}{$\alpha_1$}
\psfrag{a2}{$\alpha_2$}
\psfrag{r0}{$\rho_0$}
\psfrag{r1}{$\rho_1$}
\psfrag{r2}{$\rho_2$}
\psfrag{r3}{$\rho_3$}
\psfrag{zz}{$z$}
\psfrag{z}{$^z$}

\begin{figure}[t]
\centering
 \includegraphics[scale=.5]{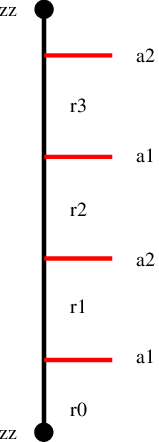}
  \caption{The circle $\zz$ in the case of torus boundary}
  \label{fig:zc}
\end{figure}

It turns out there is no $\Z$-grading on  $\mathcal A$, $\cfa$ or $\cfd$. Instead, the algebra and  domains are graded by a nonabelian group $G$, and the left or right modules over the algebra are graded by left or right cosets of  a subgroup of $G$. That subgroup is the image of periodic domains in $G$.  For a general discussion of gradings, see \cite[Chapter 10]{bfh2}. In our case, the elements of $G$ are quadruples of half-integers $(a; b,c; d)$ with $b+c\in \Z$ and $d\in \Z$, with multiplication given by
$$(a_1; b_1, c_1; d_1)\cdot(a_2; b_2, c_2; d_2) = (a_1+ a_2+ \Big|\begin{array}{cc}b_1 & c_1\\ b_2 & c_2\end{array}\Big|; b_1+ b_2, c_1+c_2; d_1+d_2).$$
The first number is called the Maslov component of the grading, and the pair $(b,c)$ is the $spin^c$ component. The fourth number is used in the case of knots to encode the $U$ grading.

The grading on  $\mathcal A$ is given by 
\begin{align*}
gr(\rho_1) & = (\textstyle -\frac{1}{2}; \frac{1}{2}, -\frac{1}{2}; 0)\\
gr(\rho_2) & = (\textstyle -\frac{1}{2}; \frac{1}{2}, \frac{1}{2}; 0)\\
gr(\rho_3) & = (\textstyle -\frac{1}{2}; -\frac{1}{2}, \frac{1}{2}; 0).
\end{align*}

For a homology solid torus, hence for both $Y_{p,1}$ and $Y_{K, n}$, the group of periodic domains is isomorphic to $\Z$, and so is its image in $G$.

For $\cfd(Y_{K,n})$ we find a generator $h$ for this image in Section \ref{dfrom-}. If $D_I$ is a coefficient map from $x$ to $y$ then the gradings of $x$ and $y$ are related by
\begin{equation}\label{eqn:cfd}
gr(y) = \lambda^{-1}gr(\rho_I)^{-1}gr(x)\in G/\h,
\end{equation}
where $\lambda = (1; 0, 0; 0)$. 

For $\cfa(Y_{p,1})$, we find a generator $g$ for the subgroup in Section \ref{cfa}. For a multiplication map $m_{l+1} (x, \rho_{I_1}, \ldots, \rho_{I_l}) = U^i y$ we have the formula
\begin{equation}\label{eqn:cfa}
gr(y) = \lambda^{l-1}gr (x)gr(\rho_{I_1})\cdots gr(\rho_{I_l})(0; 0,0; i)\in \g\backslash G.
\end{equation}

Following the notation in \cite{bfh2}, we denote the associated graded objects to $\cfkm$ and $\cfkhat$ by $g\cfkm$ and $g\cfkhat$, and the homologies of $g\cfkm$ and $g\cfkhat$ by $\hfkm$ and $\hfkhat$.
Let $\HH$ be a provincially admissible Heegaard diagram for $Y_{p,1}$ with $\bdy\HH = \zz$. While $\cfam(\HH)$ may not be an invariant of $Y_{p,1}$(Remark $11.20$ of \cite{bfh2}), the pairing theorem  \cite[Theorem 11.21]{bfh2} says that there are homotopy equivalences
\begin{align*}
g\cfkm(K_{p, pn+1})&\simeq \cfam(\HH)\boxtimes \cfdhat(Y_{K,n})\\
g\cfkhat(K_{p, pn+1})&\simeq \cfahat(\HH)\boxtimes \cfdhat(Y_{K,n}),
\end{align*}
which respect gradings in the following sense. 
The tensor product is graded by the double-coset space $\g\backslash G/\h$ via $gr(xy) = gr(x)gr(y)$ (we use the notation $xy$ to mean $x\boxtimes y$).
This double-coset space is in turn isomorphic to $\Z \times \Z$, and for a homogeneous $xy$ we can always choose a coset representative for $gr(xy)$ of the form $(a; 0, 0; d)$, where $a,d\in \Z$. We can achieve this by multiplying any other representative by appropriate powers of $g$ to the left and $h$ to the right. From there, we recover the absolute Maslov and Alexander grading by the formula
\begin{align*}
A &= d - p\tau(K) - \frac{np(p-1)}{2}\\
M &= a+ 2A.
\end{align*}
We discuss this formula  in Section \ref{tensor}.


\section{ \texorpdfstring{$\cfdhat$ of a thin knot}{CFD of a thin knot}}

\subsection{The complex $\cfkm$ for thin knots}\label{basis}

Recall that given a knot $K$ in $S^3$, $\cfkm(K)$ is a free, finitely generated chain complex over $\F_2 [U]$, endowed with an Alexander filtration $A$ by the integers, and an integer grading, called the Maslov grading. The differential lowers the Maslov grading by one, respects the Alexander filtration, and  does not decrease the $U$ power. We can illustrate $\cfkm(K)$ graphically as follows. We choose a basis of  generators $\mathcal B$ for $\cfkm(K)$ over $\F_2[U]$ which is homogeneous with respect to the Alexander filtration. Then $\mathcal B\otimes_{\F_2} \F_2[U]$ is a basis for $\cfkm(K)$ over $\F_2$.
We plot $\mathcal B\otimes \F_2[U]$ on  the $(U, A)$-lattice, and draw arrows for the differential $\bdy_-$. To match preexisting conventions, a generator of the form $U^x\xi$  of Alexander depth $y$ is at position $(-x,y)$, where $\xi \in \mathcal B$.


If $\bdy_-(x) = y_1+\cdots + y_n$, where $x, y_1, \ldots, y_n$ are elements of $\mathcal B\otimes_{\F_2} \F_2[U]$, then there is an arrow from $x$ to each $y_i$. In this case we say that $x$ \textit{points} to each $y_i$. If $y_i$   is below/to the left of $x$, we say that the arrow from $x$ to $y_i$ points \textit{down}/\textit{to the left}. Note that all arrows point non-strictly down and to the left. If the arrow is \emph{vertical}, meaning that $x$ and $y_i$ have the same $U$ power, then the \emph{length} of the arrow is $A(x)-A(y_i)$. If the arrow is \emph{horizontal}, meaning that $x$ and $y_i$ are in the same Alexander filtration, then the \emph{length} of the arrow is the difference between the $U$ power of $y$ and the $U$ power of $x$.

From now on, $K$ will be a thin knot. In this case, $\Delta_K(t)$ and $\tau(K)$ are sufficient to describe a model for the chain complex $\cfkm(K)$. Note that for a $\sigma$-thin knot, this means that the only information we need is  the Alexander polynomial and the signature. This was stated without proof in \cite{alt} with regard to alternating knots. We now state and prove the general claim.

\begin{theorem} \label{squares}
If $K$ is a  thin knot, $\cfkm(K)$ is completely determined by $\tau(K)$ and $\Delta_K(t)$.
\end{theorem}

The proof relies on two lemmas. First we perform a filtered chain homotopy to obtain a new complex with a simpler differential. Then we change basis to show that the complex is isomorphic to a direct sum of three special kinds of complexes.

\begin{lemma}
There is a filtered chain homotopy equivalence $$(\cfkm(K), \bdy_-)\cong (\hfkhat(K)\otimes \F_2[U], \bdy_z+U\bdy_w),$$
where $\bdy_z$ counts holomorphic disks that pass once through the basepoint $z$, and $\bdy_w$ counts disks that pass once through $w$.
\end{lemma}

\begin{proof} In each vertical column of the $(U, A)$-lattice, the arrows that go between elements in the same position count disks that do not pass through either basepoint, and hence form the differential $\hat\bdy$.  We take homology with respect to these arrows. In terms of basis elements, if $ \bdy_-(a) = b_1+\cdots + b_n$, and $a$ and $b_1$ have the same Alexander filtration and $U$ power, and if $x_1, \ldots, x_k$ are all the other elements that point to $b_1$, then we replace the basis vectors  $b_1, b_2, \ldots, b_n, x_1, \ldots, x_k$ with $b_1+\cdots + b_n, b_2, \ldots, b_n, x_1+a,\ldots, x_k+a$. In this way, we get an isolated arrow from $a$ to $b_1+\cdots + b_n$, so we can delete it. Repeating this until there are no more such arrows, we get a complex with generators $\hfkhat(K)\otimes \F_2[U]$. 

Since $K$ is thin, the difference in the Maslov gradings of any two generators of $\hfkhat$ is equal to the difference in their Alexander filtrations. Thus, if an arrow pointing from $x$ to $U^ly$ drops the Alexander filtration by $k$, then $M(x) - M(y) = A(x)-A(y) = k-l$, since multiplication by $U$ drops the Alexander filtration by $1$. On the other hand, since the differential always drops the Maslov grading by $1$, and multiplication by $U$ drops it by $2$, then $1 = M(x) - M(U^ly) = k-l +2l = k+l$. Then either $k=0$ and $l=1$, or $k=1$ and $l=0$.  In the first case we have a horizontal arrow of length one pointing to the left and contributing to $U  \bdy_w$, and in the second case we have a vertical arrow of length one pointing down and contributing to $\bdy_z$. 
\end{proof}

\begin{defn} A free, finitely generated, chain complex  $\mathcal C$ over $\F_2[U]$ is automatically endowed with a $U$-power filtration. An Alexander filtration $A$ is a filtration such that
\begin{itemize}
\item multiplication by $U$ lowers the $A$ filtration by $1$,
\item the differential respects $A$.
\end{itemize}
The complex $\mathcal C$ is said to be \emph{thin} if the differential  lowers the sum of $A$ and $U$-power filtration by exactly $1$. 
\end{defn}

This definition is equivalent to saying that in the graph of $\mathcal C$ all arrows are either vertical or horizontal and have length one. 

Given a thin complex, call the map consisting of all vertical arrows $\bdy_z$, and the map consisting of all horizontal arrows $U\bdy_w$. We choose this notation in order to be consistent with the case of a knot Floer complex. For a homogeneous element $x$, 
$$\bdy^2 x = (\bdy_z+U\bdy_w)^2 x = \bdy_z^2 x+U^2\bdy_w^2x +(\bdy_z(U\bdy_w)+(U\bdy_w)\bdy_z) x,$$ 
where the three homogeneous summands have distinct positions on the lattice. Since $\bdy^2=0$, then all three summands must be identically zero, showing that the maps $\bdy_z$ and $U\bdy_w$ are differentials. 

The \emph{vertical complex} $\mathcal C^{\mathrm{vert}} := \mathcal C/ (U\cdot \mathcal C)$ is a chain complex which inherits the Alexander filtration from $\mathcal C$. We call its homology the \emph{vertical homology}, denoted $H^{\mathrm{vert}}(\mathcal C)$. We also define the \emph{horizontal complex} $\mathcal C^{\mathrm{horz}}$ to be the degree zero part of the associated graded space to $\mathcal C\otimes_{\F_2[U]} \F_2[U, U^{-1}]$ with respect to the Alexander filtration. It is filtered by the $U$ powers, and inherits a differential from $\mathcal C$. We call its homology the \emph{horizontal homology}, denoted $H^{\mathrm{horz}}(\mathcal C)$. 

When $\mathcal C\cong \cfkm(K)$, then $\bdy_z$ and $\bdy_w$ are the differentials for $\cfhat(S^3)$ with respect to the two different basepoints. In that case 
$$C^{\mathrm{vert}}(\cfkm(K))\cong C^{\mathrm{horz}}(\cfkm(K))\cong \cfhat(S^3),$$ and 
$$H^{\mathrm{vert}}(\cfkm(K))\cong H^{\mathrm{horz}}(\cfkm(K))\cong \hfhat(S^3)\cong \F_2.$$

\begin{lemma} Suppose  $\mathcal C$ is a thin complex with horizontal and vertical homologies of rank at most $1$. Then $\mathcal C$ is  isomorphic to a direct sum of complexes, each modeled by one of the complexes in  Figure \ref{fig:cfkm}. In particular, $(\hfkhat(K)\otimes \F_2[U], \bdy_z+U\bdy_w)$ has a model complex isomorphic to a direct sum of these model complexes.
\end{lemma}

\psfrag{x1}{$x_1$}
\psfrag{x2}{$x_2$}
\psfrag{x3}{$x_3$}
\psfrag{x4}{$x_4$}
\psfrag{w}{$^w$}
\psfrag{a_1}{$a_1$}
\psfrag{a_2}{$a_2$}
\psfrag{a2l}{$a_{2|l|}$}
\psfrag{a2l+1}{$a_{2|l|+1}$}
\psfrag{a2|l|}{$a_{2l}$}
\psfrag{a2|l|+1}{$a_{2l+1}$}
\psfrag{C}{$C$}
\psfrag{Cl>=0}{$C_l$ for $l \leq 0$}
\psfrag{Cl<0}{$C_l$ for $l > 0$}

\begin{figure}[h]
  \centering
 \includegraphics[scale=0.5]{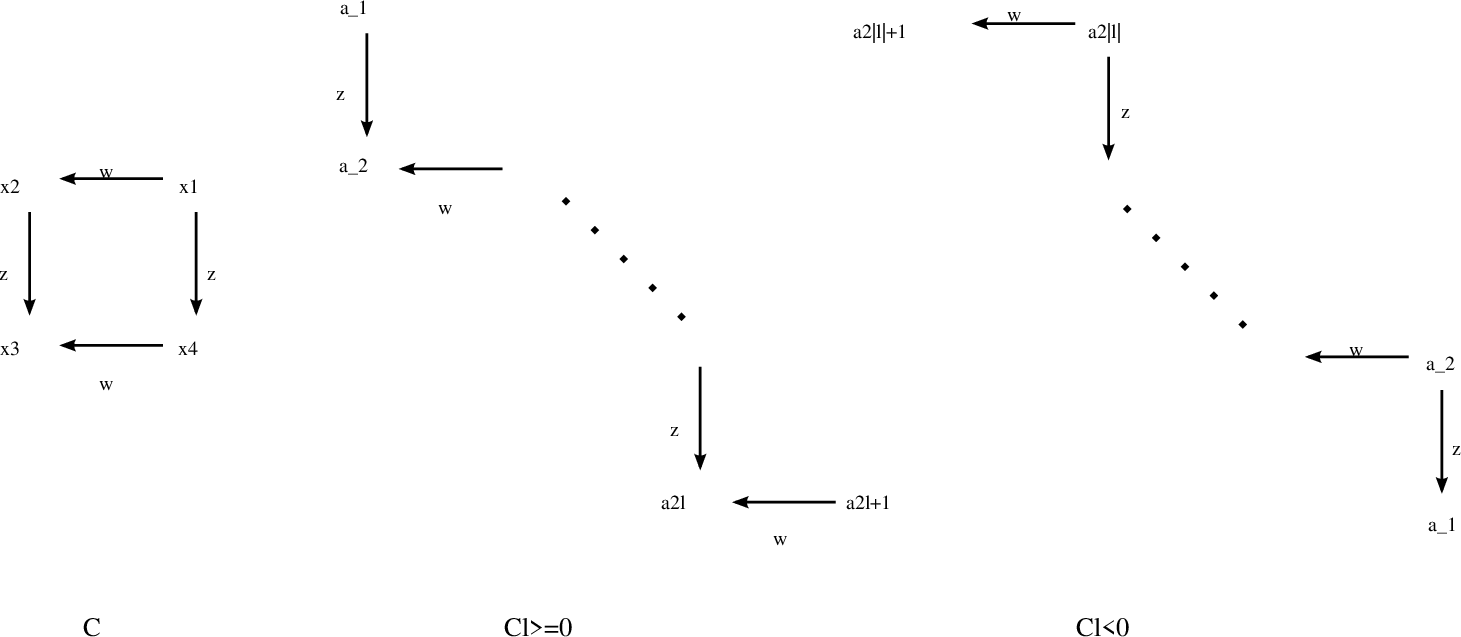}
  \caption{Model complexes for $\cfkm$ of a thin knot $K$, where $l = \tau(K)$. A $w$-arrow from $x$ to $y$ means that $y$ has coefficient $U$ in $\bdy x$.}
  \label{fig:cfkm}
\end{figure}

\begin{proof} For the sake of simplicity, we slightly abuse notation in this proof. We will say there is a $w$-arrow from $x$ to $y$, to mean that there is a $w$-arrow from $x$ to $Uy$. We will also denote $U\bdy_w$ simply by $\bdy_w$, and thus say $\bdy_w x = y$, instead of $U\bdy_w x = Uy$.

We prove the lemma by induction on $\mathrm{rk}_{\F_2[U]} (\mathcal C)$, since $\mathcal C$ has finite rank over $\F_2[U]$. We change basis in $\mathcal C$ over $\F_2[U]$ homogeneously to split off a $C$ or  a $C_i$ summand.  Then $\mathcal C \cong C\oplus \mathcal C'$ or $\mathcal C \cong C_i\oplus \mathcal C'$, where $\mathcal C'$ has lower rank than $\mathcal C$, and vertical and horizontal homologies of rank at most $1$, hence $\mathcal C'$ must split in the desired way too. Thus, $\mathcal C$ splits into a direct sum of the model complexes by induction.

On the $(U, A)$-lattice, the complex  $\mathcal C$ is supported in a strip of finite width and slope $1$. Choose a nonzero basis element $b_1$ over $\F_2[U]$ of smallest Alexander filtration possible (so $b_1$ is on the lower boundary edge of the strip).

{\bf{\bigskip}{\noindent}{\underline{Case 1:}}} There is a vertical arrow pointing to $b_1$.

Let $a$ be a generator that has a $z$-arrow to $b_1$. If $\bdy_z a = b_1$, let $b= b_1$, and if $\bdy_z a= b_1+\cdots +b_n$ and $n>1$, change basis by replacing $b_1, \ldots, b_n$ with $b = b_1+\cdots + b_n, b_2, \ldots, b_n$. Now $\bdy_z a = b$. By our choice of $b_1$, and since $\bdy^2 = 0$, we know that there is no $w$-arrow pointing to $a$ or $b$, and no $z$-arrow originating at $b$. 
If there are other generators with a $z$-arrow to $b$, add $a$ to each of them, so that in the new basis only $a$ has a $z$-arrow to $b$.

{\bf{\bigskip}{\underline{Case 1.1:}}} $\bdy_w b \neq 0.$ 

We will split off a $C$ summand. 
Since $(\bdy_z\bdy_w+\bdy_w\bdy_z) a = 0$, we have $ \bdy_wa = c\neq 0$, $\bdy_zc =  \bdy_w b$. By changing basis if necessary, we may assume $c$ is a basis element. Since $\bdy_z^2c = 0$, and $a$ is the only generator with $b$ in the image of its $\bdy_z$ differential, it follows that $a$ does not appear in $\bdy_z c =  \bdy_w b$. Thus, we may change basis if necessary  so that $\bdy_w b = d$ with $d$ a basis element, without affecting the choices made so far. Now $\bdy_w^2 = 0$ implies that $\bdy_w d = 0$, and $\bdy_z^2= 0$ implies that $\bdy_z d = 0$. 

For any other $b'$ that has a $w$-arrow to $d$, replace it by $b+b'$, so that $b$ remains the only generator with a $w$-arrow to $d$. In the same way we arrange that $c$ is the only generator with a $z$-arrow to $d$. Now $\bdy_z^2= 0 $ implies that no $z$-arrow points to $c$. After the last two changes, $a$ may no longer be the only generator with a $z$-arrow to $b$. 

Suppose there is some $a'\neq a$ with a $z$-arrow to $b$. Since $b$ is the only generator with a $w$-arrow to $d$, and $c$ is the only one with a $z$-arrow to $d$, then $\bdy_z\bdy_w+\bdy_w\bdy_z = 0 $ implies that $a'$ also points to $c$. Similarly, if $a'$ points to $c$, it must also point to $b$. Add $a$ to all such $a'$, so that $a$ is the only generator with a $z$-arrow to $b$, and the only one with a $w$-arrow to $c$. From $\bdy^2=0$  it follows that nothing points to $a$.

 \psfrag{a}{$a$}
 \psfrag{b}{$b$}
 \psfrag{c}{$c$}
 \psfrag{d}{$d$}
 
Thus, we have changed basis to split off a $C$, modeled by the square 
\begin{figure}[h]
\centering
 \includegraphics[scale=.5]{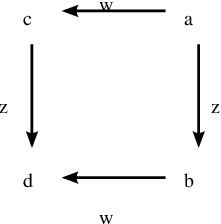}
\end{figure}

{\bf{\bigskip}{\underline{Case 1.2:}}} $\bdy_w b = 0.$

We will split off a $C_l$ summand. 
Add $a$ to any other generator that has a $z$-arrow to $b$, so that now only $a$ does. Now no $z$-arrow points to $a$, since $\bdy_z^2 = 0$.  Note that $\bdy_w b = 0$ implies that $b$ survives in horizontal homology. By the rank assumption in this Lemma, $H^{\mathrm{horz}}(\mathcal C)\cong\F_2$, represented by $b$, and so no other generator survives in horizontal homology. In particular $a$ does not survive, so  $\bdy_w a\neq 0$. As before, we may assume that $\bdy_w a = c$, where $c$ is a basis element. Note that $\bdy_w c = \bdy_z c = 0$.

Suppose that some $a_1\neq a$ has a $w$-edge to $c$, and add $a_1$ to all other such generators except $a$. Now only $a$ and $a_1$ have a $w$-edge to $c$. If $a_1$ also has $w$-edges to generators other than $c$, change basis as before to arrange that $\bdy_w a_1 = c+c_1$, where $c_1$ is a basis element. We can continue until we get a zig-zag, i.e. basis elements $a, a_1, \ldots, a_n$ with $\bdy_w a = c$, $\bdy_w a_1 = c+c_1$, $\bdy_w a_2 = c_1 + c_2, ..., \bdy_w a_{n-1} = c_{n-2}+c_{n-1}$, and either $\bdy_w a_n = c_{n-1}+ c_n$, or $\bdy_w a_n = c_{n-1}$, so that no other $w$-edge points to any $c_i$. In the first case, we replace the basis vectors $c, c_1, \ldots, c_n$ with $c, c+c_1, c_1 + c_2, \ldots, c_{n-1}+ c_n$, and in the second, we get a contradiction to the fact that the horizontal homology has rank one. Now only $a$ has a $w$-edge to $c$.

{\bf{\bigskip}{\underline{Case 1.2.1:}}} If no $z$-arrow points to $c$, then we split off the $C_1$ staircase
  
 \psfrag{a}{$a$}
 \psfrag{b}{$b$}
 \psfrag{c}{$c$}
  \begin{figure}[h]
\centering
 \includegraphics[scale=.5]{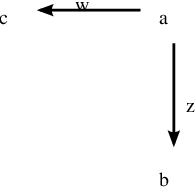}
\end{figure}

 {\bf{\bigskip}{\underline{Case 1.2.2:}}} If there is a $z$-arrow pointing to $c$, we may assume that in fact only one basis element $d$ has a $z$-arrow to $c$. 
  
 If $\bdy_z d \neq c$, then we may arrange that $\bdy_z d = c+ c^1$,  where $c^1$ is another basis element. As before, we can get a zig-zag $\bdy_z d = c+c^1$, $\bdy_z d_1 = c^1+c^2, \ldots, \bdy_z d_{k-1} = c^{k-1} + c^{k}$ and either $\bdy_z d_k = c^k + c^{k+1}$, or $\bdy_z d_k = c^k$, so that $d_i$ are in the basis and no other $z$-arrow points to any of the $c^i$. In the first case, we replace the basis vectors $c, c^1, \ldots, c^{k+1}$ by $c, c+c^1, c^1+c^2, \ldots, c^k + c^{k+1}$, and we split off the $abc$ staircase, i.e., a $C_{-1}$. In the second, we change basis by adding all $d_i$ to $d$, so that only $d' = d+d_1+\cdots + d_k$ has a $z$-arrow to $c$. Then there is no $w$-arrow to $d'$, so we can repeat the steps of Case 1.2, beginning at $d'$ instead of $a$. 
 
 If $\bdy_z d = c$, we can repeat the steps of Case 1.2, beginning at $d$ instead of $a$.

 Since the complex is supported in a diagonal strip of finite width, eventually we have to stop, and we split off a staircase $C_l$ for some $l>0$. 

{\bf{\bigskip}{\noindent}{\underline{Case 2:}}} There is no vertical arrow pointing to $b = b_1$.

We will split off a $C_l$ summand. If $\bdy_w b = 0$, then we split off a single $b$. Otherwise, we may assume that $\bdy_w b = c$, where $c$ is a basis element. Add $b$ to any other $b'$ with a $w$-arrow to $c$, so that now only $b$ has a $w$-arrow to $c$. Since $\bdy^2 b=0$, then $\bdy_w c =\bdy_z c = 0$. Since there is no $z$ arrow to $b$, then the vertical homology is $\F_2$, represented by $b$, so there is some $d$ with a $z$-arrow to $c$, and we may assume that $d$ is a basis vector. Add $d$ to all other $d'$ that have a $z$-arrow to $c$, so that now only $d$ does. Since $\bdy^2 = 0$, there is no $w$-arrow to $d$. We can proceed as in Case 1.2.2. Eventually we split off a $C_l$ for some $l \leq 0$.
\\

In each of the cases we managed to split off a model complex, so by induction on the rank of $\mathcal C$ over $\F_2[U]$, we are done.

In the special case of $(\hfkhat(K)\otimes \F_2[U], \bdy_z+U\bdy_w)$, both the vertical and horizontal homologies have rank $1$. Hence, the complex splits into exactly one $C_l$ summand, and possibly multiple $C$ summands.
\end{proof}

\begin{proof}[Proof of Theorem \ref{squares}]
We showed there is an isomorphism
$$\oplus_{i=1}^k (C'_i)\cong (\hfkhat(K)\otimes \F_2[U], \bdy_z+U\bdy_w)$$
for some $k$, where each $C'_i$ is one of the model complexes in Figure \ref{fig:cfkm}. If we restrict to the vertical column of the $(U, A)$-lattice where the $U$-power is zero, we see exactly one representative over $\F_2[U]$ of each generator of $\oplus_{i=1}^k C'_i$. For each square $C$, its representatives in this column appear in three adjacent Alexander gradings, with two representatives in the middle grading. For the staircase $C_l$, its representatives in the column appear one in each of $2|l|+1$ adjacent gradings. Also note that this column is isomorphic to $\hfkhat(K)$, so its rank  in any Alexander grading $a$ equals the rank of $\hfkhat(K)$ in the same Alexander grading $a$ (which also equals the absolute value of the coefficient of the symmetrized Alexander polynomial in degree $a$). Figure \ref{fig:5_2} illustrates these observations for the knot $5_2$.

 \psfrag{...}{\reflectbox{$\ddots$}}
  \psfrag{a_3}{$a_3$}

\begin{figure}[h]
  \centering
 \includegraphics[scale=0.5]{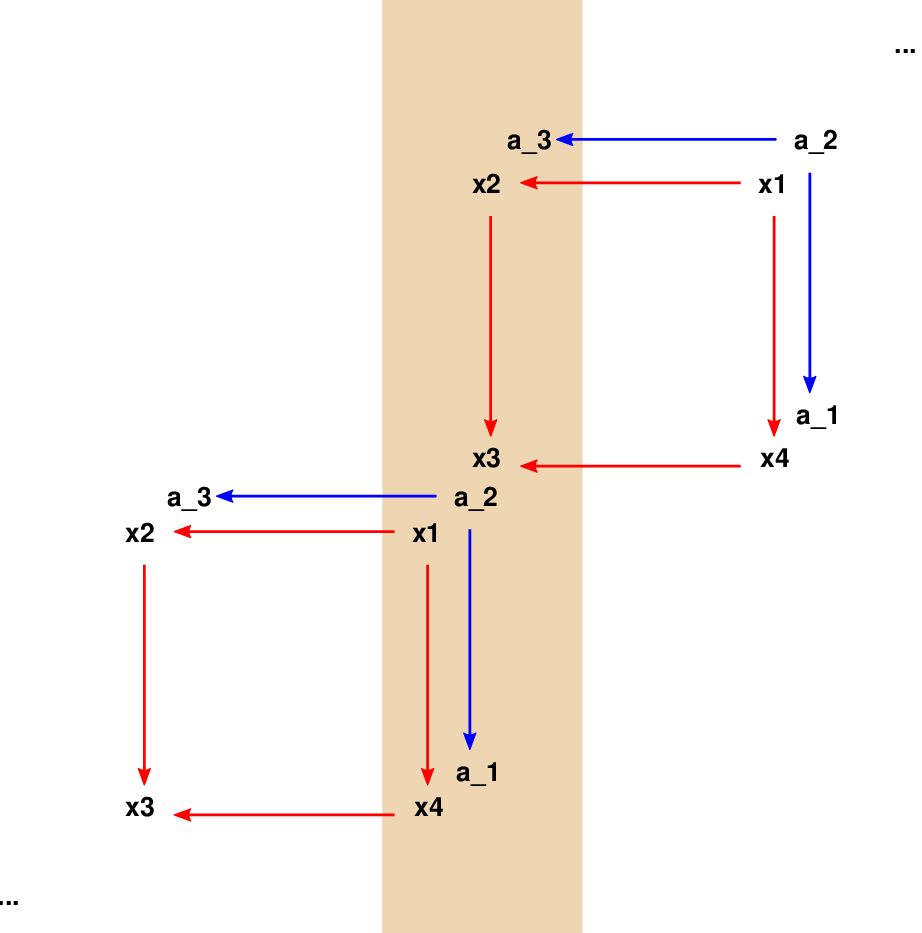}
  \caption{$\cfkm$ of the $5_2$ knot. The $U^0$ column is highlighted. To keep the figure simple, we omit writing the $U$-power of the generators when translating them by the $U$-action.}
  \label{fig:5_2}
\end{figure}

The two ends of the staircase are generators for $\widehat{HF}(S^3)$, $a_1$ with respect to the basepoint $z$, and $a_{2|l|+1}$ with respect to the basepoint $w$. Thus, the representative for $a_1$ has Alexander grading $-\tau(K)$, and the one for $a_{2|l|+1}$ has Alexander grading $\tau(K)$ (see Section \ref{tau} for the definition of $\tau$). Then the staircase looks like $C_{\tau(K)}$. It contains $2|\tau(K)|+1$ elements, one in each Alexander grading $i$, where $-|\tau(K)|\leq i\leq |\tau(K)|$.

Let $a_i'$ be the rank of the column in Alexander grading $i$ after removing all the staircase generators. In other words,

\begin{displaymath} 
a_i' = \left\{ \begin{array}{ll} 
|a_i| & \textrm{if $|i| > |\tau(K)|$}\\ 
|a_i|-1 & \textrm{otherwise,}
\end{array} \right. 
\end{displaymath} 
where $a_i$ is the coefficient of $t^i$ in the symmetrized Alexander polynomial $\Delta_K(t)$. Let $c_i$ be the number of squares with an upper right corner representative in Alexander grading $i$. We see that $c_{g-1} = a_g'$, $c_{g-2} = a_{g-1}' - 2c_{g-1}$, and in general we get the recursive formula $c_i = a_{i+1}' - 2c_{i+1} - c_{i+2}$. Note, in particular, that $c_i = c_{-i}$. 
\end{proof}

\subsection{ \texorpdfstring{$\cfdhat$ from $\cfkm$}{CFD from CFK-}}\label{dfrom-}

Theorems $11.27$ and A.11 of \cite{bfh2} together provide an algorithm for computing $\cfdhat$ of any bordered knot complement with framing $n$ from $\cfkm$. In particular, if $K$ is thin, we take the simplified basis described in Section \ref{basis} and modify each square and the one staircase as in Figure \ref{fig:tocfd}. To simplify our notation when working with indices and gradings, we will often write $\tau$ for $\tau(K)$ when it is clear from the context what we mean.

\psfrag{d1}{$^{D_1}$}
\psfrag{d2}{$^{D_2}$}
\psfrag{d3}{$^{D_3}$}
\psfrag{d123}{$^{D_{123}}$}
\psfrag{y1}{$y_1$}
\psfrag{y2}{$y_2$}
\psfrag{y3}{$y_3$}
\psfrag{y4}{$y_4$}

\psfrag{a2l}{\hspace{-.2cm}$a_{2|\tau|}$}
\psfrag{a2l+1}{$a_{2|\tau|+1}$}
\psfrag{a2|l|}{$a_{2\tau}$}
\psfrag{a2|l|+1}{\hspace{4pt}$a_{2\tau+1}$}

\psfrag{u1}{$u_1$}
\psfrag{u2}{$u_2$}
\psfrag{v1}{$v_1$}
\psfrag{v2}{$v_2$}
\psfrag{u2l}{\hspace{-.2cm}$u_{2|\tau|}$}
\psfrag{u2l+1}{\hspace{-1pt}$u_{2|\tau|+1}$}
\psfrag{v2l}{$v_{2|\tau|}$}
\psfrag{v2l-1}{\hspace{-.2cm}$v_{2|\tau|-1}$}
\psfrag{u2|l|}{$u_{2\tau}$}
\psfrag{u2|l|+1}{$u_{2\tau+1}$}
\psfrag{v2|l|}{\hspace{.1cm}$v_{2\tau}$}
\psfrag{v2|l|-1}{$v_{2\tau-1}$}

  The dashed diagonal arrow $\dashrightarrow$  stands for
\begin{displaymath}
\xymatrix{ 
\ar[r]^{D_{12}} & & & & & &  \text{if }  n = 2\tau, \hspace{62pt}\\
\ar[r]^{D_1}&\mu_1  &  \mu_2 \ar[l]_{D_{23}}& \dots\ar[l]_{D_{23}} & \mu_m \ar[l]_{D_{23}}& \ar[l]_{D_3} &  \text{if } n = 2\tau-m, \quad  m>0,\\
\ar[r]^{D_{123}}&\mu_1 \ar[r]^{D_{23}} &  \mu_2 \ar[r]^{D_{23}}& \dots\ar[r]^{D_{23}}& \mu_{|m|} \ar[r]^{D_2}&  &  \text{if } n = 2\tau-m, \quad  m<0.
}
\end{displaymath}

Note that when $\tau(K)=0$ and $n\geq 2\tau(K)$, the type $D$ structure is not bounded. In that case, to obtain a bounded $\cfdhat(K, n)$  we modify the dashed arrow to
\begin{displaymath}
\xymatrix{ 
\ar[r]^{D_{1}} &\epsilon_1 &\ar[l]_{1} \epsilon_2  \ar[r]^{D_{2}}& & & & &\hspace{-20pt} \text{if }  n = 2\tau, \hspace{60pt}\\
\ar[r]^{D_{1}} &\epsilon_1 &\ar[l]_{1} \epsilon_2  \ar[r]^{D_{23}}&\mu_1 \ar[r]^{D_{23}} & \dots\ar[r]^{D_{23}}& \mu_{|m|} \ar[r]^{D_2}&  & \hspace{-20pt} \text{if } n = 2\tau-m, \hspace{.2cm}  m<0.
}
\end{displaymath}

\begin{figure}[H]
\vskip 1 cm
  \centering
\hspace{-22pt}   \includegraphics[scale=0.4]{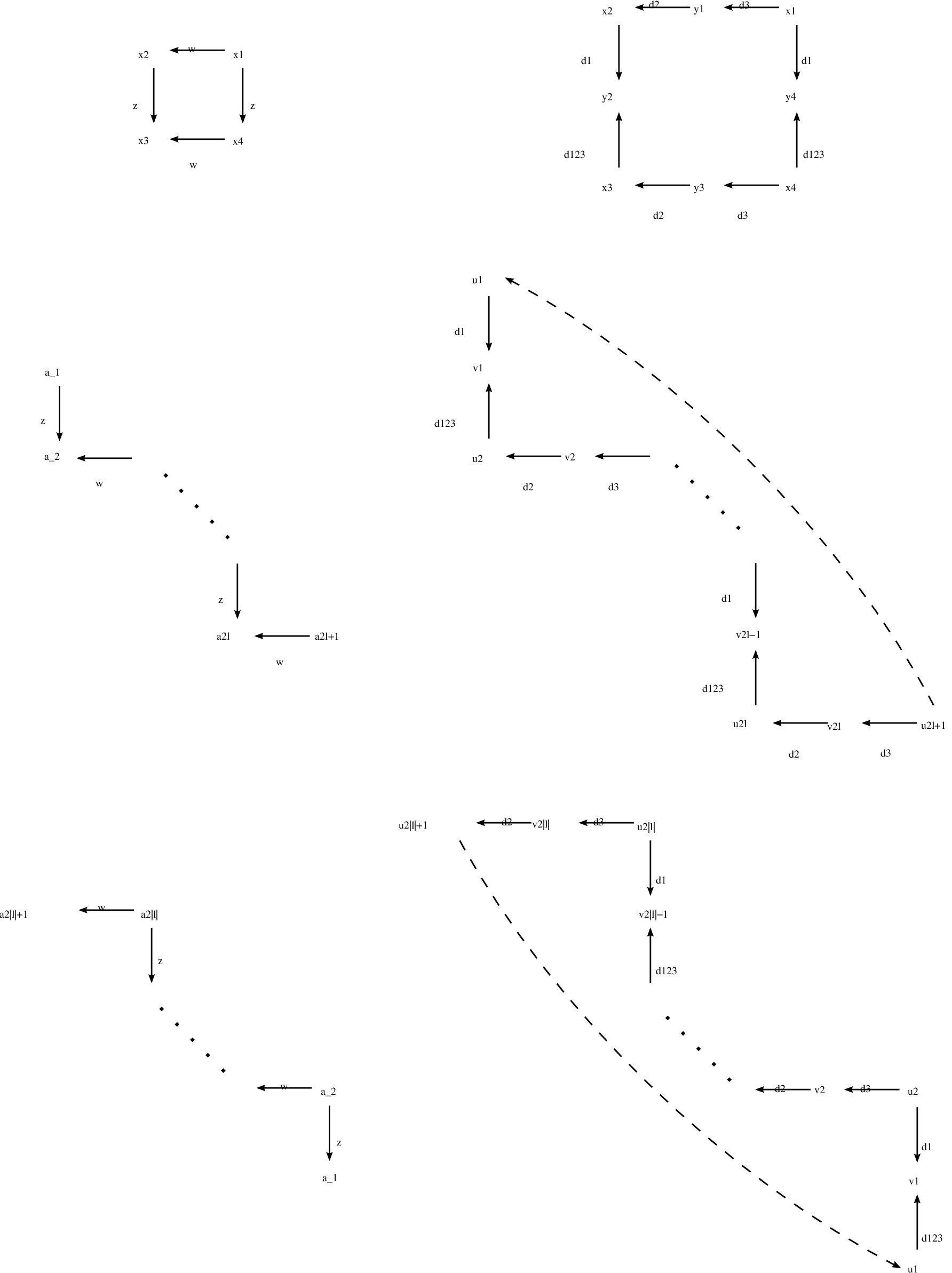}
  \caption{From $\cfkm$ to $\cfdhat$. The table shows each model knot Floer complex on the left, and the corresponding type $D$ module on the right.}
  \label{fig:tocfd}
\end{figure}

Next, we find the gradings of the elements of $\cfdhat (K, n)$. Recall that for a homology solid torus, such as a knot complement, the group of periodic domains is isomorphic to $\Z$ (see, for example,  the discussion above  \cite[Lemma 11.40]{bfh2}), so working with base generator $u_1$, the image of this group $\pi_2(u_1, u_1)$ in $G$ has a single generator $h$. Thus, $\cfdhat (K, n)$ is graded by $G/\h$. We normalize the grading by setting $gr(u_1) = (0; 0,0; 0)/\h$. Starting at $u_1$ and using (\ref{eqn:cfd}) and the grading on the algebra, we  go along the staircase, then along the dashed arrow:

{\bf{\bigskip}{\noindent}{\underline{Case 1:}}} If $\tau(K) \leq 0$, the   staircase $C_{\tau}$ is graded as follows
\begin{align*}
gr(u_{2k+1}) &= (k; 0, 2k; 0)/ \h \\
gr(v_{2k+1}) &= (-\textstyle\frac{1}{2}; -\frac{1}{2}, 2k+\frac{1}{2}; 0)/\h  \\
gr(u_{2k}) &= (k-\textstyle\frac{1}{2}; 0, 2k-1; 0)/ \h \\
gr(v_{2k}) &= (2k-\textstyle\frac{1}{2}; \frac{1}{2}, 2k-\frac{1}{2}; 0)/ \h.
\end{align*}
If $m>0$, we have the extra elements $\mu_1, \ldots, \mu_m$, with gradings
$$gr(\mu_{i+1}) = (i-\textstyle\frac{1}{2}; -\textstyle\frac{1}{2}, i+\textstyle\frac{1}{2}-2\tau; 0)/ \h.$$
If $m\leq 0$ and $\tau(K) = 0$, we have $\epsilon_1$ and $\epsilon_2$ graded as
\begin{align*}
gr(\epsilon_1) &= (-\textstyle\frac{1}{2}; -\textstyle\frac{1}{2}, \textstyle\frac{1}{2}; 0)/ \h \\
gr(\epsilon_2) &= (\textstyle\frac{1}{2}; -\textstyle\frac{1}{2},  \textstyle\frac{1}{2}; 0)/ \h,
\end{align*}
and for any $\tau(K)\leq 0$ the additional elements $\mu_1, \ldots, \mu_{|m|}$ when $m<0$, with gradings
$$gr(\mu_{i+1}) = (-i-\textstyle\frac{1}{2}; -\textstyle\frac{1}{2}, -i-\textstyle\frac{1}{2}-2\tau; 0)/ \h.$$
In each case, by closing the loop back at $u_1$ along the dashed arrow, we see that the grading of $u_1$ is also given by 
$$gr(u_1) = (\textstyle\frac{m}{2}-\textstyle\frac{1}{2}+\tau; -1, m-2\tau; 0)/\h.$$
The difference $(\textstyle\frac{m}{2}-\textstyle\frac{1}{2}+\tau; -1, m-2\tau; 0)$ of the two grading representatives then lies in $\h$. Since this difference is primitive, it equals $h$ or its inverse, so we choose
$$h = (\textstyle\frac{m}{2}-\textstyle\frac{1}{2}+\tau; -1, m-2\tau; 0).$$

{\bf{\noindent}{\underline{Case 2:}}} If $\tau(K) > 0$, the staircase  $C_{\tau}$ is graded as follows
\begin{align*}
gr(u_{2k+1}) &= (-k; 0, -2k; 0)/ \h \\
gr(v_{2k+1}) &= (-\textstyle\frac{1}{2}; -\frac{1}{2}, -2k-\frac{1}{2}; 0)/\h  \\
gr(u_{2k}) &= (-k+\textstyle\frac{1}{2}; 0, -2k+1; 0)/ \h \\
gr(v_{2k}) &= (-2k+\textstyle\frac{1}{2}; \frac{1}{2}, -2k+\frac{1}{2}; 0)/ \h,
\end{align*}
and $h$ and the gradings of the extra elements for each framing are given by the same formula as in the $\tau(K)\leq 0$ case.

To compute the gradings of all squares, we rely on the following lemma.

\begin{lemma}\label{cfdgr}
All elements  of $\cfkm$ on a fixed line of slope $1$ on the  $(U, A)$-lattice are converted to elements of the same grading in $\cfdhat$. In fact, if $x$ and $y$ are the generators of $\cfdhat$ in idempotent $\iota_0$ corresponding to $x'$ and $y'$ in $\cfkm$, and if 
$$M(x')-M(y') = n = A(x')-A(y'),$$ then the relative $G/\h$ grading of $x$ and $y$ is given by
$$gr(y) = (\textstyle\frac{n}{2}; 0, n; 0)gr(x).$$
\end{lemma}

\begin{proof}
The Lemma follows directly from \cite{bfh2}. 
The changes of bases in the proofs of \cite[Theorems 11.27, 11.35, and 11.37]{bfh2} all respect gradings, so it suffices to verify the statement for generators in a pair of Heegaard diagrams $\mathcal H_K$ and $\mathcal H(n)$, as in \cite[Figure 11.8]{bfh2}.  There is only one $Spin^c$ structure for $S^3$, so $\pi_2(x', y')$ is nonempty. Take any domain $D$ from $x'$ to $y'$, and add to it $-n_z(D)$ copies of the Heegaard surface for $\mathcal H_K$, to obtain a domain $B'$ from $x'$ to $y'$ that misses the basepoint $z$. From
\begin{align*}
M(x')-M(y') &= \textrm{ind}(B') - 2n_w(B') = n\\
A(x')-A(y') &= n_z(B') - n_w(B') = n
\end{align*}
it follows that $n_w(B') = -n$ and $\textrm{ind}(B')= -n$. In the bordered diagram $\mathcal H(n)$, there is a corresponding domain $B$ from $x$ to $y$ which crosses the boundary regions labeled by $2$ and $3$ with multiplicity $-n$ each. We make use of the more general grading theory for a moment, and work with the grading group $G'(4)$ \cite[Section 10]{bfh2}. From \cite[Equations 10.2, 10.19, and 10.27]{bfh2}, 
\begin{align*}
R(g'(B)) &= (-e(B)-n_x(B) - n_y(B); r_{\ast}(\bdy^{\bdy}(B)))\\
gr'(y) &= R(g'(B))gr'(x).
\end{align*}
Observe that
\begin{align*}
e(B)+n_x(B) + n_y(B) &= e(B')+\frac{n}{2}+n_{x'}(B') + n_{y'}(B')\\
&=\textrm{ind}(B')+\frac{n}{2}\\
&= -\frac{n}{2},
\end{align*}
so  $gr'(y) = (\frac{n}{2}; 0, n, n)gr'(x)$. Switching back to the grading group $G$, $gr(y) = (\frac{n}{2}; 0, n; 0)gr(x)$.
\end{proof}

We say that a square of $\cfdhat$ lies in level $t$ if the upper right corner of the corresponding small square of $\cfkm(K)$ is on a line of slope $1$ that is $t$ units below the line through $a_1$, i.e., the upper right corner element has Maslov grading $-2\tau(K)-t$ as an element of $\hfkhat$. Note that $t$ can be negative, meaning that the square is above the $a_1$-line. By Theorem \ref{squares}, there  are $c_{t+\tau(K)}$ squares in level $t$.
By Lemma \ref{cfdgr}, each square in level $t$ has upper right corner $x_1$ in grading $(\frac{t}{2}; 0, t; 0)/\h$, and using (\ref{eqn:cfd}) again, the grading of the whole square is given by
\begin{align*}
gr(x_1) & = (\textstyle \frac{t}{2}; 0, t; 0)/\h & gr(y_1) & = (t \textstyle -\frac{1}{2}; \frac{1}{2},t -\frac{1}{2}; 0)/\h\\
gr(x_2) & = (\textstyle \frac{t}{2}-\frac{1}{2}; 0, t-1; 0)/\h & gr(y_2) & = (\textstyle -\frac{1}{2}; -\frac{1}{2}, t-\frac{1}{2}; 0)/\h\\
gr(x_3) & = (\textstyle \frac{t}{2}; 0, t; 0)/\h & gr(y_3) & = (t+\textstyle \frac{1}{2}; \frac{1}{2}, t+\frac{1}{2}; 0)/\h\\
gr(x_4) & = (\textstyle \frac{t}{2}+\frac{1}{2}; 0, t+1; 0)/\h & gr(y_4) & = (\textstyle -\frac{1}{2}; -\frac{1}{2}, t+ \frac{1}{2}; 0)/\h.
\end{align*}


\section{$\cfam$ of the $(p,1)$-cable in the solid torus}\label{cfa}

Figure \ref{fig:p1} shows a bordered Heegaard diagram for the $(p, 1)$-cable in the solid torus.

\psfrag{b_1}{$b_1$}
\psfrag{b_p-2}{$b_{p-2}$}
\psfrag{b_p-1}{$b_{p-1}$}
\psfrag{b_p}{$b_p$}
\psfrag{b_p+1}{$b_{p+1}$}
\psfrag{b_2p-2}{$b_{2p-2}$}
\psfrag{a}{$a$}
\psfrag{1}{$1$}
\psfrag{2}{$2$}
\psfrag{3}{$3$}
\psfrag{z}{$z$}
\psfrag{w}{$w$}
\psfrag{...}{$...$}

\begin{figure}[H]
  \centering
 \includegraphics[scale=0.6]{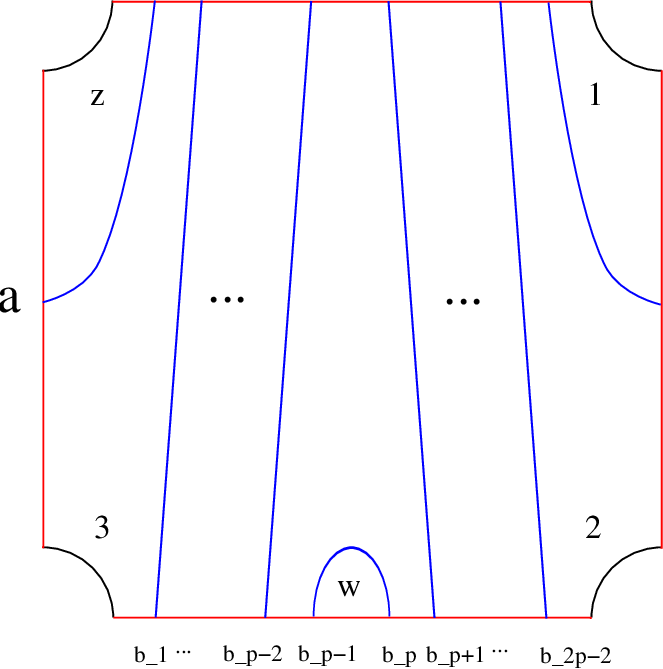}
  \caption{The $(p,1)$-cable in the solid torus}
  \label{fig:p1}
\end{figure}

 The module $\cfam(Y_{p,1})$ is generated over $\F_2[U]$ by $a, b_1, \ldots, b_{2p-2}$. The multiplication maps count certain $J$-holomorphic curves in $\Sigma\times [0,1]\times \R$, whose relative homology class  has index 1.  For more detail on the moduli spaces, indices, and expected dimension for bordered diagrams, see \cite[Section 5]{bfh2}.  Since the Heegaard surface $\Sigma$ is a punctured torus, we can instead count embedded disks of index $1$ in  its universal cover $\mathbb C \setminus (\Z\times \Z)$  connecting lifts of generators and missing the preimage of $z$, modulo vertical and horizontal translations of the lattice. The positive periodic domains are generated by the domain $B_a$ in Figure \ref{fig:pd}.  It is straightforward to enumerate the finitely many embedded disks whose boundary does not project to all of $\beta$. Any other disk that contributes to the multiplication maps is a sum of one of these and a positive number of copies of $B_a$. Figure \ref{fig:pd} shows the periodic domain, and the only domain that contributes to $m_4(b_6,\rho_2, \rho_{12}, \rho_1)$ on a portion of the universal cover for the $(4,1)$-cable.
 
 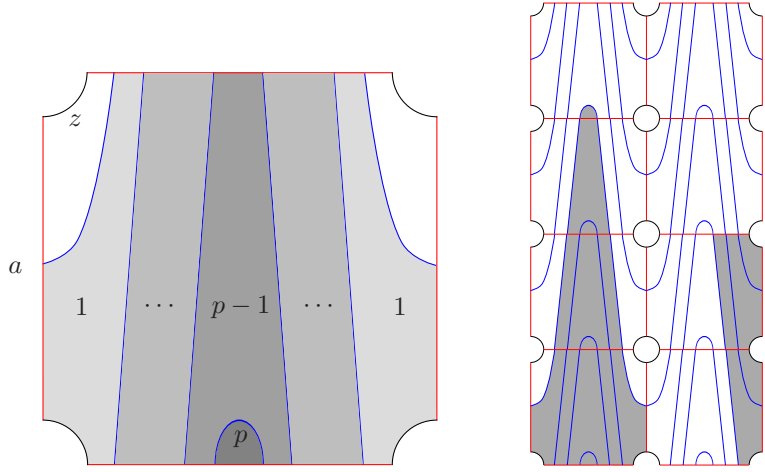
\begin{figure}[H]
  \centering
  \labellist
       \pinlabel $z$ at 25 260
       \pinlabel $a$ at -20 150
       \pinlabel $1$ at 30 120
       \pinlabel $\cdots$ at 90 120
       \pinlabel $\cdots$ at 210 120
       \pinlabel $1$ at 270 120
       \pinlabel $p-1$ at 150 120
       \pinlabel $p$ at 150 20
           \endlabellist
 \includegraphics[scale=0.5]{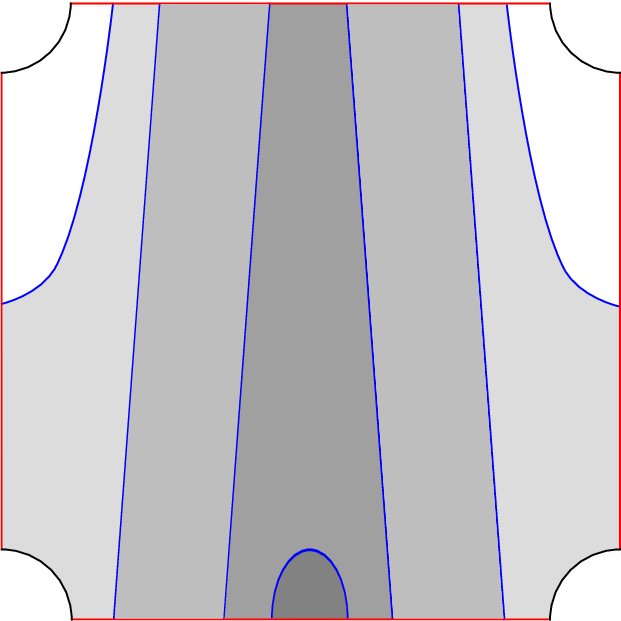}
 \hspace{1cm}
  \includegraphics[scale=0.55]{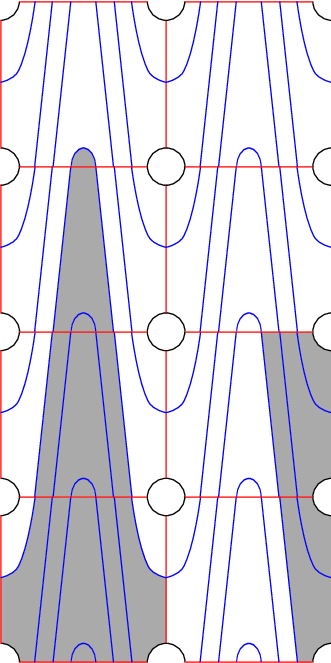}
  \caption{Left: The generator $B_a$ for positive periodic domains. Right: The lift of the Heegaard diagram to the universal cover of the punctured torus in the case of the $(4,1)$-cable, along with $B_a$ and the domain for $m_4(b_6,\rho_2, \rho_{12}, \rho_1)$.}
  \label{fig:pd}
\end{figure}

  We conclude  that the only multiplication maps are
\begin{align*}
m_1(b_k) &= U^{p-k}\cdot b_{2p-k-1} \quad &\text{f}&\text{or }  1\leq k\leq p-1 \\
m_{3+i}(b_k,\rho_2, \underbrace{\rho_{12}, \ldots, \rho_{12}}_i, \rho_1) & = U^{i+1}\cdot b_{k+i+1}   &\text{f}&\text{or } \begin{array}{l}1\leq k\leq p-2,\\
 0\leq i\leq p-k-2
 \end{array}\\
 m_{3+i}(b_k,\rho_2, \underbrace{\rho_{12}, \ldots, \rho_{12}}_i, \rho_1) & = b_{k-i-1}  &\text{f}&\text{or }  
 \begin{array}{l} p+1\leq k\leq 2p-2,\\
 0\leq i\leq k-1-p
 \end{array}\\
m_{2+i}(a,\underbrace{\rho_{12}, \ldots, \rho_{12}}_i, \rho_1) & = b_{2p-i-2}  &\text{f}&\text{or } 0\leq i\leq p-2\\
m_{4+i+j}(a,\rho_3, \underbrace{\rho_{23}, \ldots, \rho_{23}}_j, \rho_2, \hspace{1cm}
 & = U^{pj+i+1}\cdot b_{i+1}  &\text{f}&\text{or } 
\begin{array}{l}0\leq i\leq p-2,\\
 0\leq j
 \end{array}\\
\underbrace{\rho_{12}, \ldots, \rho_{12}}_i, \rho_1) & & & \\
 m_{3+j}(a, \rho_3, \underbrace{\rho_{23}, \ldots, \rho_{23}}_j, \rho_2 ) &= U^{p(j+1)}\cdot a \quad &\text{f}&\text{or }  0\leq j \\
\end{align*}
Setting $U=0$ yields $\cfahat(Y_{p,1})$ - the generators are $a, b_1, \ldots, b_{2p-2}$, and the multiplication maps are given by the third and fourth map above.

Applying (\ref{eqn:cfa}) to the last multiplication map for $j=0$, we find an indeterminacy of $(-\frac{1}{2}; 0,1; p)$ for the grading of $a$, i.e. the element $g = (-\frac{1}{2}; 0,1; p)$ is in the image of $\pi_2(a,a)$ in $G$. Since $g$ is primitive in $G$, it is a generator for that image. 

Next, we normalize the $\g\backslash G$ grading on the generators by setting
$$gr(a) := \g\backslash(0; 0, 0; 0). $$ 
From the fourth multiplication map we see that 
$$gr(b_{2p-i-2}) = \g\backslash(-\textstyle\frac{1}{2}; i+\frac{1}{2},-\frac{1}{2}; 0) \quad \text{for }  0\leq i\leq p-2, $$
and from the first map get
$$gr(b_i) = \g\backslash(\textstyle\frac{1}{2}; i-\frac{1}{2}, -\frac{1}{2}; i-p) \quad \text{for }  1\leq i\leq p-1.$$
We remind the reader that an alternate choice of multiplication maps for this computation may at first seem to provide different gradings, but one can verify that those grading representatives lie in the same coset as the ones provided above.

\note We can construct a Heegaard diagram similar to the one in Figure \ref{fig:p1} to study any cable $K_{p,q}$. Given any relatively prime integers $p$ and $q$, with $p>1$,  let $i$ be the unique integer $1\leq i< p$, $i = q \textrm{ mod } p$. One can construct a genus $1$ bordered diagram for the $(p, i)$-cable as follows. Theorem 3.5 of  \cite{phil} provides an algorithm for constructing genus $1$ Heegaard diagrams for $(1,1)$ knots. In particular, the construction for the diagram of the  torus knot $T_{p,i}$ can be modified to provide a bordered diagram for the $(p, i)$-cable in the solid torus. We first find the Heegaard normal form of the attaching curve $\beta$ (Figure 3.11 of \cite{phil} demonstrates this process for $T_{5,3}$) so that the corners of the fundamental domain for the flat torus coincide with the basepoint $w$. Then we draw two $\alpha$-circles through the basepoint $w$, one identified with the two vertical segments of the boundary of the domain, hence isotopic to the meridian of the torus, the other identified with the two horizontal segments of the boundary, hence isotopic to the standard longitude and to $\beta$, and remove a neighborhood of $w$. To be consistent with our conventions, we rename $z$ to $w$, and place a new basepoint $z$ in the bottom rightmost region of the diagram. Since the diagram has genus $1$, it is straightforward to compute the multiplication maps for $\cfa^-(p,i)$. By tensoring with $\cfdhat(K, (q-i)/p)$ as in Section \ref{tensor}, we can compute $\hfkhat(K_{p,q})$ and $\tau(K_{p,q})$.


\section{\texorpdfstring{The tensor product $\cfam \boxtimes \cfdhat$}{The tensor product of CFA and CFD}} \label{tensor}

We now compute the gradings of the tensor product in $(N, A')$ notation, and list the differentials. Then we find the shifting constant $c$.

Since $\cfdhat(K,n)$ splits as a direct sum of squares and a staircase, its tensor product with $\cfam(Y_{p,1})$ splits as a direct sum of the tensor products of $\cfam(Y_{p,1})$ with each square and the staircase too. For this reason, we refer to the corresponding direct summands of the tensor product as squares and a staircase too.

Each of the $c_{t+\tau(K)}$ squares in level $t$ is graded by
\begin{align*}
gr(ax_1) &= (t,-p t)\\
gr(ax_2) &= (-1+t,p-p t)\\
gr(ax_3) &= (t,-p t)\\
gr(ax_4) &= (1+t,-p-p t)\\
gr(b_k y_1) &= (-2 k+2 k \tau-k n-k^2 n+t+2 k t, k+k n p-p t)\\
gr(b_{2p-i-2} y_1)  &= (-3-\! 2 i + 2 \tau + 2 i \tau-\!  2 n- \! 3 i n-\!  i^2 n+3 t+2 i t, p +n p+i n p-\!  p t)\\
gr(b_k y_2)  &= (1-2 k-2 \tau+2 k \tau+k n-k^2 n-t+2 k t, k-n p+k n p-p t)\\
gr(b_{2p-i-2} y_2)  &= (-2-2 i+2 i \tau-i n-i^2 n+t+2 i t, p+i n p-p t)\\
gr(b_k y_3)  &= (1+2 k \tau-k n-k^2 n+t+2 k t, k-p+k n p-p t)\\
gr(b_{2p-i-2} y_3)  &= (-2 l+2 i \tau-2 n-3 i n-i^2 n+3 t+2 i t, n p+i n p-p t)\\
gr(b_k y_4) &= (2 l+2 k \tau+k n-k^2 n-t+2 k t, k-p-n p+k n p-p t)\\
gr(b_{2p-i-2} y_4)  &= (-1+2 i \tau-i n-i^2 n+t+2 i t, i n p-p t),
\end{align*}
 where $1\leq k\leq p-1$ and $0\leq i\leq p-2$.

  Matching up the coefficient maps of $\cfdhat(K,n)$ with the multiplication maps in the beginning of this Section, we see that the non-trivial  differentials on each square are
   \begin{align*}
     \bdy(ax_1)  &= b_{2p-2} y_4 + U b_1y_2 + U^p ax_2 \quad  &    \bdy(b_i y_1) &=  U^{p-i} b_{2p-i-1} y_1 + U b_{i+1}y_2 \\
    \bdy(ax_2)  &= b_{2p-2} y_2    & \bdy(b_{p-1} y_1) &= U b_p y_1 & \\
   \bdy (ax_4)  &= U^p ax_3        &  \bdy(b_{2p-i-1} y_1) &=   b_{2p-i-2} y_2 \\
\bdy(b_k y_j) &=  U^{p-k} b_{2p-k-1} y_j
\end{align*}
 where $1\leq i \leq p-2$, $j = 2, 3, 4$ and $1 \leq k\leq p-1$.
 
The gradings and differentials on the staircase, including the diagonal string connecting $a_1$ and $a_{2|\tau|+1}$, depend on $\tau(K)$:

 {\bf{\bigskip}{\noindent}{\underline{Case 1:}}} If $\tau(K)\leq 0$, the gradings on the staircase are 
\begin{align*}
gr(au_{2t+1}) &= (2 t, -2 p t)\\
gr(au_{2t}) &= (-1+2 t, p-2 p t)\\
gr(b_k v_{2t+1}) &= (-2 \tau+2 k \tau+k n-k^2 n-2 t+4 k t, k+k n p - n p - p -2 p t) \phantom{\hspace {10pt}}\\
gr(b_k v_{2t}) &= (-2 k+2 k \tau-k n-k^2 n+2 t+4 k t, k+k n p-2 p t)\\
 gr(b_{2p-i-2} v_{2t+1}) &= (-1+2 i \tau-i n-i^2 n+2 t+4 i t, i n p-2 p t)\\
gr(b_{2p-i-2} v_{2t}) &= (-3-2 i+2 \tau+2 i \tau-2 n-3 i n-i^2 n+6 t+4 i t, \\
& \qquad p+n p+i n p-2 p t)
\end{align*}
For  $ m = 2\tau(K) - n > 0$, 
\begin{align*}
gr(b_{2p-i-2}\mu_{j+1}) &= (-1+2j+2ij- 2\tau- 2i\tau   - in- i^2n,  -jp + 2 p \tau + inp) \phantom{\hspace {1cm}}\\
gr(b_k\mu_{j+1}) &= (2jk - 2k \tau + kn- k^2n , k - p - j p  +2p \tau - np + knp),
\end{align*}
and for $m<0$,
\begin{align*}
gr(b_{2p-i-2}\mu_{j+1}) &= (-2  - 2i - 2j  -2ij- 2\tau - 2i\tau   - in- i^2n,  p + jp   + 2p \tau + inp) \\
gr(b_k\mu_{j+1}) &= (1- 2k -2jk-2k\tau  + kn-  k^2n ,  k+ j p  +2p \tau- np + knp  ).
\end{align*}
When $\tau(K) = 0$ and $m\leq 0$, we also have 
\begin{align*}
gr(b_k\epsilon_1) &= (k n-k^2 n, k-p-n p+k n p)\phantom{\hspace {4.5cm}} \\
gr(b_k\epsilon_2) &= (1+k n-k^2 n, k-p-n p+k n p)\\
gr(b_{2p-i-2}\epsilon_1) &= (-1-i n-i^2 n, i n p)\\
gr(b_{2p-i-2}\epsilon_2) &= (-i n-i^2 n, i n p).
\end{align*}

{\bf{\bigskip}{\noindent}{\underline{Case 2:}}} If $\tau(K) > 0$, the gradings on the staircase are
\begin{align*}
gr(au_{2t+1}) &= (-2 t, 2 p t)\\
gr(au_{2t}) &= (1-2 t, -p+2 p t)\\
gr(b_k v_{2t+1}) &= (1-\! 2 k-\! 2 \tau+2 k \tau+k n-k^2 n+2 t-4 k t, k-n p+k n p+2 p t)\\
gr(b_k v_{2t}) &= (1+2 k \tau-k n-k^2 n-2 t-4 k t, k-p+k n p+2 p t)\\
gr(b_{2p-i-2} v_{2t+1}) &= (-2-2 i+2 i \tau-i n-i^2 n-2 t-4 i t, p+i n p+2 p t)\\
gr(b_{2p-i-2}v_{2t}) &= (2 \tau+2 i \tau-2 n-3 i n-i^2 n-6 t-4 i t, n p+i n p+2 p t)
\end{align*}
 The gradings on all $b_i\mu_j$  are given by the same formula as in Case 1.

Next, we list the non-trivial differentials.
\noindent
When $\tau<0$ and  $m>0$, the non-trivial differentials are
\begin{align*}
\bdy(au_1) &= b_{2p-2}v_1  & \bdy(b_k v_s) &=  U^{p-k} b_{2p-k-1} v_s  \phantom{\hspace{1.6cm}}\\
 \bdy(au_{2t+1}) &= b_{2p-2}v_{2t+1} + U^p au_{2t} \phantom{\hspace{0.8cm}} & \bdy(b_k \mu_j) &=  U^{p-k} b_{2p-k-1}\mu_j    \\
 \bdy(au_{2|\tau|+1}) &= b_{2p-2}\mu_1 + U^p au_{2|\tau|} &    
\end{align*}
When $\tau<0$ and  $m=0$ the non-trivial differentials are
\begin{align*}
 \bdy(au_1) &= b_{2p-2}v_1  &   \bdy(au_{2|\tau|+1}) &= b_{2p-3}v_1 + U^p au_{2|\tau|} \phantom{\hspace{1.75cm}} \\
\phantom{\hspace{.3cm}}  \bdy(au_{2t+1}) &= b_{2p-2}v_{2t+1} + U^p au_{2t}  \phantom{\hspace{.5cm}} &
  \bdy(b_k v_s) &=  U^{p-k} b_{2p-k-1} v_s  
  \end{align*}
When $\tau<0$ and  $m<0$ the non-trivial differentials are
\begin{align*}
\bdy(au_1) &= b_{2p- 2}v_1  & \bdy(b_k \mu_j) &=  U^{p-k} b_{2p-k-1}\mu_j   \\
  \bdy(au_{2t+1}) &= b_{2p- 2}v_{2t+1} + U^p au_{2t}   &  \bdy( b_i\mu_{|\mspace{-1.0mu} m \mspace{-1.0mu} |} ) & = U b_{i+1}v_1 + U^{p-i}b_{2p-i- 1}\mu_{|\mspace{-1.0mu} m \mspace{-1.0mu} |}  \\
 \bdy(au_{2|\mspace{-1.0mu} \tau \mspace{-1.0mu} |+1}) &= U^p au_{2|\mspace{-1.0mu}\tau \mspace{-1.0mu} |}   & \bdy( b_{p- 1}\mu_{|\mspace{-1.0mu} m \mspace{-1.0mu} |}) & = U b_p\mu_{|\mspace{-1.0mu} m \mspace{-1.0mu} |} \\
\bdy(b_k v_s) &=  U^{p-k} b_{2p-k-1} v_s   &   \bdy( b_{2p-i- 1} \mu_{|\mspace{-1.0mu} m \mspace{-1.0mu} |}) & = b_{2p-i-2}v_1
\end{align*}

\noindent
When $\tau = 0$ and $m>0$, the non-trivial differentials are
\begin{align*}
 \phantom{\hspace{.5cm}} \bdy(au_1) &= b_{2p-2}\mu_1 & \phantom{\hspace{3cm}} \bdy(b_k \mu_j) &=  U^{p-k} b_{2p-k-1}\mu_j  \phantom{\hspace{1.5 cm}} 
\end{align*}
When $\tau = 0$ and $m=0$ the non-trivial differentials are
\begin{align*}
 \bdy(au_1) &= b_{2p-2}\epsilon_1  &   \bdy( b_{p-1}\epsilon_2) & = b_{p-1}\epsilon_1 + U b_p\epsilon_2 \\
\bdy(b_k \epsilon_1) &=  U^{p-k} b_{2p-k-1} \epsilon_1     &  \bdy( b_p\epsilon_2) & = b_p\epsilon_1  \\
 \bdy( b_i\epsilon_2) & =  b_i\epsilon_1 + U b_{i+1}\epsilon_1+ U^{p-i}b_{2p-i-1}\epsilon_2  &  \bdy( b_{2p-i-1}\epsilon_2) & =  b_{2p-i-1}\epsilon_1 + b_{2p-i-2}\epsilon_1 
\end{align*}
When $\tau = 0$ and $m<0$ the non-trivial differentials are
\begin{align*}
\bdy(au_1) &= b_{2p-2}\epsilon_1  &  \bdy(b_k \mu_j) &=  U^{p-k} b_{2p-k-\mspace{-1.0mu}1}\mu_j  \\
\bdy(b_k \epsilon_1) &=  U^{p-\mspace{-1.0mu}k} b_{2p-k-\mspace{-1.0mu} 1} \epsilon_1  &  \bdy( b_i\mu_{|\mspace{-1.0mu} m \mspace{-1.0mu} |}) & = U b_{i+\mspace{-1.0mu}1}\epsilon_1 + U^{p-\mspace{-1.0mu}i}b_{2p-i-\mspace{-1.0mu}1}\mu_{|\mspace{-1.0mu} m \mspace{-1.0mu} |} \\
 \bdy( b_k\epsilon_2) & = b_k\epsilon_1 + U^{p-\mspace{-1.0mu}k}b_{2p-\mspace{-1.0mu}k-\mspace{-1.0mu}1}\epsilon_2   & \bdy( b_{p-\!1}\mu_{|\mspace{-1.0mu} m \mspace{-1.0mu} |}) & = U b_p\mu_{|\mspace{-1.0mu} m \mspace{-1.0mu} |}  \\
\bdy( b_{2p-\mspace{-1.0mu}k-\mspace{-1.0mu}1}\epsilon_2) & = b_{2p-k-\mspace{-1.0mu} 1}\epsilon_1 & \hspace{-8pt} \bdy( b_{2p-i-\mspace{-1.0mu} 1} \mu_{|\mspace{-1.0mu} m \mspace{-1.0mu} |}) & = b_{2p-i-2}\epsilon_1  
\end{align*}

\noindent
When $\tau >0$ and  $m>0$, the non-trivial differentials are
\begin{align*}
\bdy(au_{2t}) &= b_{2p-2}v_{2t-1}+ U^pau_{2t+1}  &  \bdy(b_{2p-i-1}v_{2\tau}) &= b_{2p-i-2}\mu_1\\ 
\bdy(au_{2\tau}) &= b_{2p-2}v_{2\tau-1}+ Ub_1\mu_1+ U^p au_{2\tau+1}  &   \bdy(b_{p-1} v_{2\tau}) &= U b_p v_{2\tau} \\
\bdy(au_{2\tau+1}) &= b_{2p-2}\mu_1 & \bdy(b_k v_s) &= U^{p-k}b_{2p-k-1}v_s\\
\bdy(b_i v_{2\tau}) &= U b_{i+1}\mu_1+  U^{p-i}b_{2p-i-1}v_{2\tau}& \bdy(b_k \mu_j) &= U^{p-k}b_{2p-k-1}\mu_j \phantom{\hspace{1cm}}
\end{align*}
When $\tau >0$ and  $m=0$ the non-trivial differentials are
\begin{align*}
 \phantom{\hspace{.5cm}} \bdy(au_{2t}) &= b_{2p-2}v_{2t-1} + U^p au_{2t+1}  \phantom{\hspace{1cm}} &
\bdy(b_k v_s) &=  U^{p-k} b_{2p-k-1} v_s  \phantom{\hspace{1cm}}
 \end{align*}
When $\tau >0$ and  $m<0$ the non-trivial differentials are
\begin{align*}
 \bdy(au_{2t}) &= b_{2p-2}v_{2t-1} + U^p au_{2t+1}  \phantom{\hspace{2.5cm}} &
\bdy(b_k \mu_j) &=  U^{p-k} b_{2p-k-1} \mu_j  \phantom{\hspace{.5cm}}   \\
 \bdy(b_k v_s) &=  U^{p-k} b_{2p-k-1} v_s  
\end{align*}

The indices vary as follows
\begin{align*}
 & 1\leq s \leq \left\{ \begin{array}{ll} 2|\tau|-1 &  \text{ if } m > 0, \tau>0     \\
2|\tau| &  \text{ otherwise }
 \end{array} \right.     & 1\leq  j \leq \left\{ \begin{array}{ll} |m|-1 &  \text{ if } m < 0, \tau\leq 0     \\
|m| & \text{ otherwise }
 \end{array} \right. \\
&1\leq  t \leq \left\{ \begin{array}{ll} |\tau|  & \hspace{6pt} \text{ if } m<0, \tau >0     \\
|\tau| -1 & \hspace{6pt} \text{ otherwise }
 \end{array} \right.  & 
\begin{array}{l} 1\leq  i\leq p-2    \\
 1\leq k\leq p-1 
 \end{array}\hspace{3.52cm}
\end{align*}

Recall that the absolute $N$ grading is obtained by requiring that the homology $H_{\ast}\left(g\cfkm(K_{p, pn+1})/U=1\right)\cong \F_2$ lives in $N$ grading $0$, see the discussion in \cite[Section 11.3]{bfh2}, specifically  Equation 11.17 and the paragraph preceding it. Set $U=1$ above. 

 When $\tau> 0$, $au_1$ splits as a direct summand of the chain complex, so it represents  $H_{\ast}\left(g\cfkm(K_{p, pn+1})/U=1\right)\cong \F_2$, implying that $N(au_1) = 0$. 

When $\tau < 0$, the subcomplex $D$ generated by $au_1, b_1v_1,$ and $b_{2p-2}v_1$ splits, and 
$$H_{\ast}(D) = \frac{\ker(D)}{\text{im}(D)} = \langle au_1 + b_1v_1\rangle,$$
so $N(au_1 + b_1v_1)= 0$.

When $\tau = 0$ and $m>0$, the subcomplex $D$ generated by $au_1, b_1\mu_1,$ and $b_{2p-2}\mu_1$ splits, and 
$$H_{\ast}(D) = \frac{\ker(D)}{\text{im}(D)} = \langle au_1 + b_1\mu_1\rangle,$$
so $N(au_1 + b_1\mu_1)= 0$.

When $\tau = 0$ and $m=0$, The image and kernel  of the differential are
\begin{align*}
\text{im} &  =   \langle  b_1 \epsilon_1 + b_2 \epsilon_1 + b_{2p-2} \epsilon_2,  b_2 \epsilon_1 + b_3 \epsilon_1 + b_{2p-3} \epsilon_2 , \ldots,  b_{p-2} \epsilon_1 + b_{p-1} \epsilon_1 + b_{p+1} \epsilon_2 , \\
  & \qquad b_{p-1} \epsilon_1 + b_p \epsilon_2,   b_p\epsilon_1, b_{p+1}\epsilon_1, \ldots, b_{2p-2}\epsilon_1 \rangle \\
  \ker & =  \textrm{im}\oplus \langle au_1+b_1\epsilon_1\rangle,
 \end{align*}
 so $au_1+b_1\epsilon_1$ survives in homology, implying that $N(au_1+b_1\epsilon_1) = 0$. 

When $\tau = 0$ and $m<0$, the subcomplex $D$ generated by $au_1$, $b_1\epsilon_1$, $b_1\epsilon_2$, $b_{2p-2}\epsilon_1$, and $b_{2p-2}\epsilon_2$
splits, and 
$$H_{\ast}(D) = \frac{\ker(D)}{\text{im}(D)} = \frac{\langle au_1 + b_1\epsilon_1, b_1\epsilon_1 + b_{2p-2}\epsilon_2, b_{2p-2}\epsilon_1  \rangle}{\langle b_1\epsilon_1 + b_{2p-2}\epsilon_2, b_{2p-2}\epsilon_1  \rangle},$$
 so $N(au_1 + b_1\epsilon_1)= 0$.
 
In each case $N(au_1) = 0$,  so the first component of the grading provided earlier in this section is, in fact, the absolute $N$ grading.

Next, we find the absolute Alexander grading, by requiring that the Euler characteristic of $g\cfkhat$ is the symmetrized Alexander polynomial. The formula 
$$\Delta_{K_{p, pn+1}}(t) = \Delta_K(t^p)\Delta_{T_{p, pn+1}}(t)$$
implies that if the degree of the symmetrized $\Delta_K(t)$ is $d$, then the degree of the symmetrized $\Delta_{K_{p, pn+1}}(t)$ is 
$$ pd+\frac{(|p|-1)(|pn+1|-1)}{2} = \left\{ \begin{array}{ll} pd+\frac{np(p-1)}{2} & \quad \text{ if } n\geq 0     \\
pd-\frac{np(p-1)}{2} - p + 1 & \quad \text{ otherwise, }
 \end{array} \right.   $$
so we look for the highest relative Alexander grading in which generators survive when taking Euler characteristic, and shift to make it equal this degree. 

If $d>|\tau|$, then this highest grading is realized by the following generators  coming from all squares in level $d - \tau - 1$:
\begin{align*}
\text{ If } n>0, \quad & b_py_1 \\
\text{ If } n= 0, \quad & ax_2, b_py_1, b_{p+1}y_1, \ldots, b_{2p-2}y_1, b_py_2, b_{p+1}y_2, \ldots, b_{2p-2}y_2 \\
\text{ If } n< 0,  \quad & b_1y_2.
\end{align*}
In each case, the contribution of each square to $\chi(g\cfkhat)$ is rank $1$ and in Maslov grading $d-\tau$ mod $2$.

If $d<|\tau|$ or if $K$  is the unknot, then the highest grading is realized by staircase generators:
\begin{align*}
\text{ If }\tau>0,  n>0, \quad & b_p v_{2\tau} \\
\text{ If } \tau>0, n= 0, \quad & au_{2\tau+1}, b_pv_{2\tau}, b_{p+1}v_{2\tau}, \ldots, b_{2p-2} v_{2\tau}, b_p\mu_1, b_{p+1} \mu_1, \ldots, b_{2p-2}\mu_1 \\
\text{ If }\tau \geq 0, n< 0,  \quad & b_1 \mu_1\\
\text{ If }\tau\leq 0,  n>0, \quad & b_p \mu_{|m|} \\
\text{ If } \tau=0, n= 0, \quad & a u_1\\
\text{ If } \tau<0, n= 0, \quad & au_1, b_p\mu_{|m|}, b_{p+1}\mu_{|m|}, \ldots, b_{2p-2}\mu_{|m|}, b_p v_1, b_{p+1} v_1, \ldots, b_{2p-2}v_1 \\
\text{ If }\tau<0, n< 0,  \quad & b_1 v_1.
\end{align*}
In each case, the contribution of the staircase to $\chi(g\cfkhat)$ is rank $1$ and in Maslov grading $0$ mod $2$.

If $d = |\tau|$, then the highest grading is realized by the listed generators from squares in level $d-\tau - 1$ combined with the listed staircase generators. The rank $1$ contribution  to $\chi(g\cfkhat)$ has Maslov grading $0$ mod $2$ both for the staircase, and for each square, so there are no further cancelations.

In each of these cases,  we need to shift by the constant   $c= -p\tau - np(p-1)/2$. 
Together with the fact that $N=M-2A$, we now have a complete description of  $g\cfkm(K_{p, pn+1})$.  Setting $U=0$ in the above differentials gives $g\cfkhat(K_{p, pn+1})$.

For completeness, we include a list of the generators of $\hfkhat(K_{p, pn+1})$. Throughout, the indices will vary as follows:
\begin{align*}
p+1& \leq i \leq 2p-2 &\\
 1& \leq k\leq p &\\
  1& \leq j\leq |m|-1 &\\
1& \leq  t \leq \tau   &\textrm{ if } \tau >0 \phantom{.}   \\
0& \leq  t \leq  |\tau| -1   &\textrm{ if } \tau <0. 
 \end{align*}

For each square direct summand, all generators survive in homology except  
\begin{displaymath}
 \hspace{-23pt}\begin{array}{l@{\hspace{5em}}l}
ax_1 &  b_{i-1}y_2 \\
ax_2  & b_{2p-2}y_2  \\
  b_i y_1  &   b_{2p-2}y_4
\end{array}
\end{displaymath}

The staircase summand depends on $\tau$ and the framing:

If $\tau>0, m>0$, all generators survive except 
\begin{displaymath}
\hspace{-34pt} \begin{array}{l@{\hspace{2.1em}}l}
au_{2t}         &   b_i v_{2\tau} \\
au_{2\tau+1}  &  b_i\mu_1\\
 b_{2p-2}v_{2t-1}  & b_p \mu_1
\end{array}
\end{displaymath}

 If $\tau>0, m\leq 0$, all generators survive except 
\begin{displaymath}
\hspace{-20pt} \begin{array}{l@{\hspace{4.7em}}l}
\quad au_{2t}     &  
b_{2p-2}v_{2t-1} 
\end{array}
\end{displaymath}

If $\tau < 0, m>0$, all generators survive except 
\begin{displaymath}
\hspace{-9pt} \begin{array}{l@{\hspace{3.1em}}l}
 au_{2t+1} &       b_{2p-2}v_{2t+1} 
\\
 au_{2|\tau|+1}    &       b_{2p-2}\mu_1
\end{array}  
\end{displaymath}

If $\tau < 0, m=0$, all generators survive except 
\begin{displaymath}
\hspace{-10pt} \begin{array}{l@{\hspace{3.1em}}l}
 au_{2t+1}   &     b_{2p-2}v_{2t+1}
\\
au_{2|\tau|+1}  &     b_{2p-3}v_1
\end{array}  
\end{displaymath}

If $\tau < 0, m<0$, all generators survive except 
\begin{displaymath}
\hspace{-28pt} \begin{array}{l@{\hspace{2.1em}}l}
 au_{2t+1}& b_i \mu_{|m|}   
\\
 b_{2p-2}v_{2t+1} &  b_{i-1} v_1 
\end{array}  
\end{displaymath}

If $\tau=0$ and $m>0$, all generators survive except 
\begin{displaymath}
\hspace{-22pt} \begin{array}{l@{\hspace{5em}}l}
au_1 & 
b_{2p-2}\mu_1
\end{array}  
\end{displaymath}

If $\tau=0$ and $m=0$, the homology has rank one and is represented by 
\begin{itemize}
\setlength{\itemindent}{10em}
\item $au_1+ \sum_{s=p}^{2p-2}b_s\epsilon_2$
\end{itemize}

If $\tau=0$ and $m<0$, the homology is generated by 
\begin{itemize}
\setlength{\itemindent}{10em}
\item $au_1+b_{2p-2}\epsilon_2$
\item  $b_k\mu_{|m|}$
\item $b_{i-1}\epsilon_2 + b_i\mu_{|m|}$
\item $b_k\mu_j$
\item $b_i\mu_j$
\end{itemize}

This completes the description of $\hfkhat(K_{p, pn+1})$.


\section{\texorpdfstring{$\tau$ of the cable}{Tau of the cable}}\label{tau}

In \cite{tau}, Ozsv\'ath and Szab\'o define a  concordance invariant $\tau(K)$ arising from the Alexander filtration on $\cfhat(S^3)$. Alternatively, $\tau$ can be defined in terms of the associated graded object $\hfkm(K)$
 by
$$\tau(K) =- \max \{s | \forall d\geq 0, U^d \hfkm(K,s)\neq 0\}.$$
(see Lemma A.2 of  \cite{oszt}).

We do not need to fully compute the homology of $g\cfkm(K_{p, pn+1})$. It is enough to observe that $U^p\hfkm$ vanishes for each direct summand of the tensor product coming from a square in the $D$ module. Thus, $\tau$ of the $(p, pn+1)$ cable only depends on the staircase summand, which agrees with $g\cfkm$ of  the $(2,2\tau+ \textrm{sgn}(\tau))$-torus knot, and hence 
$$\tau(K_{p, pn+1}) = \tau((T_{2, 2\tau+ \textrm{sgn}(\tau)})_{p, pn+1}).$$
Since $|\tau(T_{2, 2\tau+ \textrm{sgn}(\tau)})| = g(T_{2, 2\tau+ \textrm{sgn}(\tau)})$, we can use \cite[Theorem 1.2]{mc2} to  determine $\tau((T_{2, 2\tau+ \textrm{sgn}(\tau)})_{p, pn+1})$. 

Alternatively, one can work out the computation independently by using the complex $g\cfkm(K_{p, pn+1})$ provided in Section \ref{tensor}. 
For example, when $\tau(K) < 0$ and $m>0$, the staircase summand splits further, and all direct summands vanish in $U^{2p}\hfkm$, except for the one generated by $au_1, b_1v_1$, and $b_{2p-2}v_1$. The homology of this summand is generated by $U^{p-1}au_1+ b_1v_1$, and survives in all powers of $U$. The Alexander gradings are 
$A(U^{p-1}au_1) = A(au_1) -2(p-1) = -p\tau -\frac{np(p-1)}{2} -2p +2$
 and $A(b_1v_1) =   -p\tau -\frac{np(p-1)}{2}-p+1$, so the Alexander filtration level of $U^{p-1}au_1+ b_1v_1$ is $ -p\tau -\frac{np(p-1)}{2}-p+1$, hence $\tau(K_{p, pn+1}) = -(-p\tau -\frac{np(p-1)}{2}-p+1)$. The computation in the remaining cases goes the same way. The only generator in $\hfkm$ that survives in all $U$-powers is $U^{p-1}au_1+ b_1v_1$ if $\tau<0$, $U^{p-1}au_1+ b_1\mu_1$ if $\tau=0$ and $n<0$, in which cases 
$$\tau(K_{p, pn+1})= p\tau (K) + \frac{np(p-1)}{2}+p-1 ,$$
 and $au_1$ if $\tau=0$ and $n\geq 0$, or if $\tau>0$, in which cases
$$\tau(K_{p, pn+1}) = p\tau (K) + \frac{np(p-1)}{2}.$$
This completes the proof of Theorem \ref{main}.
Observe that this agrees with the results in \cite{mc2}, where Hedden computes $\tau$ of  $(p, pn+1)$-cables for sufficiently large $|n|$. 


\section{Proof of corollaries}\label{cors}

We prove the two corollaries stated at the end of Section \ref{intro}.
\begin{proof}[Proof of Corollary \ref{pq}]

Since $K_{1,q} = K$ for all $q$, the result for $p=1$ is a tautology. 

Fixing $K$ and $p>1$, Van Cott \cite{cvc} defines the function 
$$h(q) = \tau(K_{p,q})- \frac{p-1}{2}q$$
with domain all integers relatively prime to $p$,
and proves that $h$ is non-increasing.

By Theorem \ref{main}, $h(pn+1)$ is constant as a function of $n$  for $\tau=0$ and $n\geq 0$, or $\tau > 0$, and is given by
$$h(pn+1) = p\tau (K) - \frac{p-1}{2},$$
 implying that for $\tau=0$ and  any $q> 0$, or $\tau > 0$ and any $q$, 
 $$h(q) = p\tau (K) - \frac{p-1}{2}.$$
 We see that
\begin{align*}
\tau(K_{p,q}) &= h(q) + \frac{p-1}{2}q\\
&= p \tau(K) + \frac{(p-1)(q-1)}{2}.
\end{align*}

It is shown that $\tau$ changes sign under reflection\cite{tau}. Thus, since $\bar{K_{p,q}} = \bar K_{p, -q}$, it follows that for $\tau(K) =0$ and $q<0$, or $\tau(K) < 0$ we have   \belowdisplayskip=-12pt
\begin{align*} 
\tau(K_{p,q}) & = -\tau(\bar K_{p,-q}) \\
& = -\left(p \tau(\bar K) + \frac{(p-1)(-q-1)}{2}\right)\\
 & = p\tau(K) + \frac{(p-1)(q+1)}{2}.  
 \end{align*} \qedhere    
\end{proof}

\begin{proof}[Proof of Corollary \ref{genus}]

Since $\tau(K)=g_4(K)\geq 0$ and $q>0$, Corollary \ref{pq} implies that
$$\tau(K_{p,q}) =  p \tau(K) + \frac{(p-1)(q-1)}{2}.$$
 Substituting $g_4(K)$ for $\tau(K)$ we see that 
$$\tau(K_{p,q}) =  p g_4(K) + \frac{(p-1)(q-1)}{2}.$$
On the one hand, we know that
$$g_4(K_{p,q}) \leq  p g_4(K) + \frac{(p-1)(q-1)}{2},$$
since we can construct a surface for $K_{p,q}$ in the four-ball by connecting $p$ parallel copies of the surface for $K$ via $(p-1)q$ twisted bands. Thus $g_4(K_{p,q})\leq \tau(K_{p,q})$.
On the other hand, $g_4(K_{p,q})\geq \tau(K_{p,q})$ for any knot, implying the desired result.
\end{proof}

 
 \section{An example} \label{ex}
 
 The author programmed the results of this paper into \emph{Mathematica} \cite{m} to be able to compute specific examples. The program takes $\Delta_K(t), \tau (K), p,$ and $n$ as input, and outputs $\hfkhat(K_{p, pn+1})$. We use the program to compute $\hfkhat$ of the $(5, 16)$-cable of the knot $11n50$. We include the relevant data for $11n50$ for the reader's convenience:
 \begin{align*}
 \Delta_{11n50}(t) &= 2t^{-2}- 6 t^{-1} + 9 - 6 t + 2t^2,\\
 \tau(11n50) &= 0.
 \end{align*}
 Note that here $ p = 5$ and $n = 3$.

 We describe $\hfkhat$ of the cable as a polynomial, where the coefficient of $x^A y^M$ is the rank of $\hfkhat$ in Alexander grading $A$ and Maslov grading $M$: 
 \begin{align*}
&  2x^{-40} y^{-78}+2 x^{40} y^2+ 2x^{-39} y^{-77}+2 x^{39} y + 4x^{-35} y^{-69} + 4 x^{35} y+4x^{-34} y^{-68}\\
&+4 x^{34} +5x^{-30} y^{-60}+  5x^{30}+5x^{-29} y^{-59}+5 x^{29}y^{-1}+x^{-25} y^{-52}
+ 2x^{-25} y^{-51}\\
&+x^{25}y^{-2}+2 x^{25}y^{-1}
   + x^{-24} y^{-51}+ 4x^{-24}y^{-50} +x^{24}y^{-3}+ 4 x^{24}y^{-2}+ 2x^{-23} y^{-49}\\
   &+2x^{23}y^{-3}+3x^{-20} y^{-44} +2x^{-20} y^{-43}+ 3 x^{20}y^{-4}+ 2x^{20} y^{-3}+ 5x^{-19} y^{-43} \\
   &+ 5 x^{19}y^{-5}+ 4 x^{-18} y^{-42}
   + 4 x^{18}y^{-6} + 2x^{-15} y^{-37}+ 3 x^{-15} y^{-36}+ 2 x^{15} y^{-7}\\
   &+3 x^{15}y^{-6}+ 4x^{-14} y^{-36}+ 4 x^{14}y^{-8}
   +5 x^{-13} y^{-35} +5 x^{13}y^{-9} + 3x^{-10} y^{-30}\\
   &+ 2x^{-10} y^{-29}+ 3 x^{10}y^{-10}+ 2 x^{10}y^{-9}+ 2 x^{-9} y^{-29}+2 x^9 y^{-11} + x^{-8} y^{-29} \\
   &+ 4 x^{-8} y^{-28}+ x^8 y^{-13}+ 4 x^8 y^{-12}+ 2 x^{-7} y^{-27}+2 x^7 y^{-13}+ 3x^{-5} y^{-24}\\
   &+2 x^{-5}  y^{-23}+3 x^5 y^{-14}+ 2 x^5 y^{-13}+5 x^{-3} y^{-23} + 5
   x^3 y^{-17}+ 4 x^{-2} y^{-22}\\
   &+ 4 x^2 y^{-18} + 2 y^{-19}+ 3 y^{-18}
   \end{align*}
   
We also plot the result on the $(A,M)$-axis (without marking the rank at each coordinate): 
\begin{figure}[h]
\centering
 \includegraphics[scale=.9]{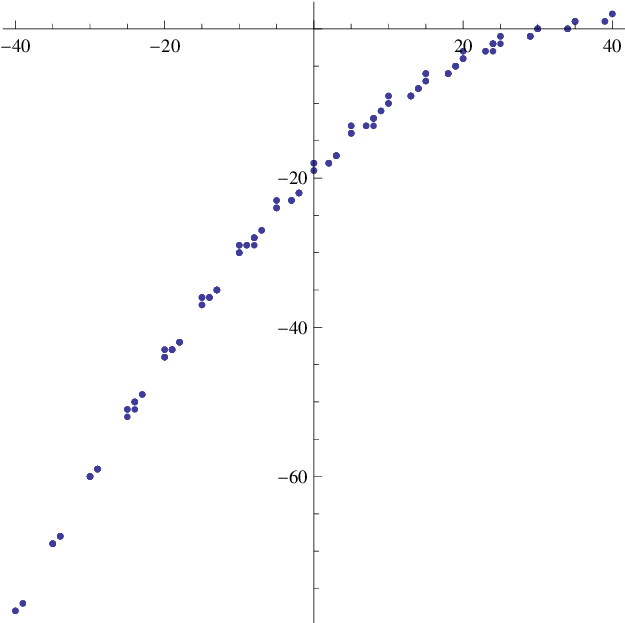}
  \caption{$\hfkhat$ of the $(5,16)$-cable of the knot $11n50$}
  \label{fig:516}
\end{figure}

\bibliographystyle{/Users/inapetkova/Documents/work/hamsplain2}

\bibliography{master}

\end{document}